\documentclass[11pt,a4paper]{article}

\usepackage[latin1]{inputenc}
\usepackage[T1]{fontenc}
\usepackage{latexsym,amsmath,amsthm,amssymb}
\usepackage{graphicx,color}
\usepackage{esint}
\usepackage{hyperref}
\usepackage{fullpage}

\newcommand{\bbR}{{\mathbb{R}}}
\newcommand{\bbN}{{\mathbb{N}}}
\newcommand{\Lie}{{\mathrm {Lie}}}
\newcommand{\diver}{{\mathrm {div}}}
\newcommand{\calA}{{\mathcal{A}}}
\newcommand{\calD}{{\mathcal{D}}}
\newcommand{\calL}{{\mathcal{L}}}
\newcommand{\calO}{{\mathcal{O}}}
\newcommand{\calQ}{{\mathcal{Q}}}
\newcommand{\calW}{{\mathcal{W}}}
\newcommand{\calX}{{\mathcal{X}}}
\newcommand{\calZ}{{\mathcal{Z}}}

\newtheorem {theorem}{Theorem}[section]
\newtheorem {proposition}[theorem]{Proposition}
\newtheorem{example}{Example}
\newtheorem {corollary}[theorem]{Corollary}
\newtheorem {definition}[theorem]{Definition}
\newtheorem{lemma}[theorem]{Lemma}
\newtheorem {remark}[theorem]{Remark}

\title{Small time asymptotic on the diagonal for H\"{o}rmander's type hypoelliptic operators}
\author{Elisa Paoli}
\begin{document}
\maketitle
\begin{abstract}
We compute the small time asymptotic of the fundamental solution of H\"{o}rmander's type hypoelliptic operators with drift, at a stationary point, $x_0$, of the drift field. We show that the order of the asymptotic depends on the controllability of an associated control problem and of its approximating system. If the control problem of the approximating system is controllable at $x_0$, then so is also the original control problem, and in this case we show that the fundamental solution blows up as $t^{-N/2}$, where $N$ is a number determined by the Lie algebra at $x_0$ of the fields, that define the hypoelliptic operator.
\end{abstract}

\section{Introduction}\label{intro}
In the following paper we will study the small time asymptotic on the diagonal of the fundamental solution to an hypoelliptic operator with drift. Let $M$ be a $n$ dimensional smooth manifold and let $f_0,f_1,\ldots,f_k$ be smooth vector fields on $M$, where $f_0$ is usually called \emph{drift} field. We will consider the following operator on $C^\infty(\bbR^+\times M)$
\begin{equation}
\frac{\partial \varphi}{\partial t}-f_0(\varphi)-\frac{1}{2}\sum_{i=1}^k f_i^2(\varphi) \qquad \forall \varphi\in C^\infty(\bbR^+\times M).
\label{eq:operator}
\end{equation}

We will assume that the vector fields $f_0,f_1,\ldots,f_k$ satisfy H\"{o}rmander's condition of hypoellipticity  (see \cite{article:Hormander}), that is
\begin{equation}
\Lie_x\{\frac{\partial}{\partial t}-f_0,f_1,\ldots,f_k\}=\bbR\times T_xM \qquad \forall x\in M.
\label{eq:hormander}
\end{equation}
Under this hypothesis the operator in \eqref{eq:operator} is hypoelliptic and, as shown for example in \cite{book:IkedaWatanabe} and \cite{article:Hormander}, it admits a well defined smooth fundamental solution, $p(t,x,y)$, for small time, that is given by the probability density of the stochastic process $\xi_t$ at time $t$, starting at $x$ at time $0$ and driven by the stochastic differential equation
%\begin{equation}
$$d\xi_t=f_0(\xi_t)dt+\sum_{i=1}^k f_i(\xi_t) \circ dw_i$$
%\label{eq:stochastic}
%\end{equation}
where $w=(w_1,\ldots,w_k)$ denotes a $k$-dimensional Brownian motion and $\circ$ denotes the integration in the Stratonovich sense.

Given a stationary point, $x_0$, of the drift field $f_0$, i.e. $f_0(x_0)=0$, the goal of this paper is to study the behaviour of the heat kernel, $p(t,x,y)$, for small time, in $x_0$. \footnote{The assumption of $f_0(x_0)=0$ is important, but can be avoided in some cases. For example, one could be interested in studying the differential operator
$$\frac{\partial}{\partial t}-f_0-\frac{1}{2}\sum_{i=1}^k (f_i^2+(\diver_{\mu} f_i)f_i)$$
that can be treated as an operator of the form \eqref{eq:operator}
$$\frac{\partial}{\partial t}-\tilde{f}_0-\frac{1}{2}\sum_{i=1}^k f_i^2$$
once we take $\tilde{f}_0=f_0+\frac{1}{2}\sum_{i=1}^k (\diver_{\mu} f_i) f_i$. The results of this paper apply also to an operator of this form and we will show the small time asymptotic of the fundamental solution on the diagonal at a point $x_0$, where $f_0(x_0)=0$ (not necessarily $\tilde{f}_0(x_0)=0$). Details for this kind of operator are written in Remark \ref{th:remark1}.
}

This problem has already been studied by many authors in the past and recent period. In particular, if the vector fields $f_1,\ldots,f_k$ are Lie bracket generating, i.e. they satisfy condition \eqref{eq:hormander} without the contribution of $f_0$, many results are already known. Indeed under this hypothesis it is well defined a distance function $d(x,y)$ determined by the fields $f_1,\ldots,f_k$. Then it was proved by L\'{e}andre in \cite{article:Lmaj} and  \cite{article:Lmin} that 
%\begin{equation}
$$\lim_{t\rightarrow 0} t\log p(t,x,y)=-\frac{1}{2} d^2(x,y)$$
%\label{eq:leandre minmax}
%\end{equation}
which generalizes a result by Varadhan \cite{art:varadhan} in the elliptic case. If moreover $f_0\equiv 0$, then if $x\neq y$ and $x$ and $y$ are not conjugate along any minimal geodesic, then for some constant $C=C(x,y)\in\bbR$
%\begin{equation}
$$p(t,x,y)=\frac{C+O(t)}{t^{n/2}} e^{-\frac{d^2(x,y)}{2t}},$$
%\label{eq:BBN}
%\end{equation}
as proved by Ben Arous \cite{art:BenAroushors} and Barilari, Boscain and Neel \cite{article:BBN}. Here we see in particular that the order of the asymptotic in $t$ is determined by the dimension of the manifold. On the other hand if $x=y$ and $f_0\equiv 0$, Ben Arous and L\'{e}andre showed in \cite{article:BAL1} and \cite{article:BAL2} that
\begin{equation}
p(t,x,x)=\frac{C+O(\sqrt{t})}{t^{\calQ/2}}
\label{eq:BAL 1,2}
\end{equation}
where $\calQ$ is the homogeneous dimension of the manifold and $C$ is a constant depending on $x$. These two examples show that the order of degeneracy of the fundamental solution for small time can reveal interesting geometric properties and depends on the structure of the manifold.

In this paper we are going to determine the order of the asymptotic for small $t$ and we will see that if $f_0$ is not identically zero the asymptotic \eqref{eq:BAL 1,2} on the diagonal at a stationary point, $x_0$, of the drift field could be more general, as it was already pointed out by Ben Arous and L\'{e}andre in \cite{article:BAL2}. The proof relies on a homogeneity property of the operator in \eqref{eq:operator} under the dilations. Indeed to understand better the behaviour of the fundamental solution $p(t,x_0,x_0)$, we will dilate the space around $x_0$ and rescale the time. We will see that, under a suitable dilation, this procedure splits the operator $\frac{\partial}{\partial t}-f_0-\frac{1}{2}\sum_{i=1}^k f_i^2$ into a principal operator
\begin{equation}
\frac{\partial}{\partial t}-\hat{f}_0-\frac{1}{2}\sum_{i=1}^k \hat{f}_i^2
\label{eq:operator princ}
\end{equation}
perturbed by a small operator that goes to zero as the space becomes larger. Here the fields $\hat{f}_0,\hat{f}_1,\ldots,\hat{f}_k$ are nilpotent approximations of the original fields $f_0,f_1,\ldots,f_k$ around the point $x_0$. In this procedure it is very important to find a good coordinate system around $x_0$, so that the dilation of the space produces an hypoelliptic principal operator \eqref{eq:operator princ}. This is done by noting that the field $f_0$ acts in the operator just one time, while the fields $f_1,\ldots, f_k$ are applied twice. So we will give to the drift field a double weight with respect to the others. We explain this procedure in detail in Section \ref{sec:2}.

We will prove that the asymptotic of the fundamental solution depends on the controllability of some associated control problems: consider the control problem 
\begin{equation}
\dot{x}(t)=f_0(x(t))+\sum_{i=1}^k u_i(t) f_i(x(t))
\label{eq:intro1}
\end{equation}
where $u=(u_1,\ldots,u_k)\in L^\infty (\bbR^+;\bbR^k)$ are the controls. If the control problem \eqref{eq:intro1} is not controllable around the point $x_0$, i.e. from $x_0$ we can not reach a neighborhood of $x_0$ using curves described by the control problem, then 
$$p(t,x_0,x_0)=0 \qquad \forall t>0.$$
We will consider also the control problem induced by the approximating system %Instead, if the control problem is controllable at $x_0$, then we study the control problem induced by the approximating system
\begin{equation}
\dot{x}(t)=\hat{f}_0(x(t))+\sum_{i=1}^k u_i(t) \hat{f}_i(x(t)).
\label{eq:intro2}
\end{equation}
If it is controllable, then so is also the original one \eqref{eq:intro1}, and, if we denote by $q_0(t,x,y)$ the fundamental solution of the approximating hypoelliptic operator in \eqref{eq:operator princ}, we prove in Theorem \ref{th:order} that
\begin{equation}
p(t,x_0,x_0)=\frac{q_0(1,x_0,x_0)+O(t)}{t^{N/2}},
\label{eq:order}
\end{equation}
where the factor $t^{N/2}$ comes from the change of coordinates given by the dilations and it is determined by the order of homogeneity under the dilations of a volume form like $dx_1\wedge\ldots\wedge dx_n$ around $x_0$.

The number $N$ is an integer depending on the structure of the Lie algebra generated by the fields $f_0,f_1,\ldots, f_k$ in $x_0$. It is computed in detail in Section \ref{sec:4}, but to anticipate here how it is found, we give a brief explanation of it: for $j\in\bbN$, let
$$d_j:=\dim\; \mathrm{span}_{x_0}\left\{[f_{i_1},[\ldots,[f_{i_{l-1}},f_{i_l}]\ldots]] \mbox{ such that } \#\{i_h>0\}+2\#\{i_h=0\}\leq j\right\}$$
be the dimension of the space generated by all the Lie brackets at $x_0$ of at most $j$ vector fields $f_0,f_1,\ldots, f_k$, where the field $f_0$ is counted two times. Then the number $N$ is defined as
$$N=\sum_{j\geq 1} j\cdot d_j.$$
For an operator coming from a sub-Rimemannian manifold with no drift field, $N$ is equal to the homegeneous dimension $\calQ$ of the manifold. If the drift field is not zero, then $N$ could be more general.

In the intermediate case in which the approximating control problem \eqref{eq:intro2} is not controllable in $x_0$, but the original control problem \eqref{eq:intro1} is still controllable, then the behaviour of the asymptotic can be more general. It could vanish, or blow up faster than $t^{-N/2}$ and even exponentially fast, as it was already pointed out by Ben Arous and L\'{e}andre in \cite{article:BAL2}.

The paper is organized as follows. We will begin by describing in details in Section \ref{sec:2} the homogeneity properties of the operator in \eqref{eq:operator}. In particular we will derive the conditions that the dilations have to satisfy in order to produce the right split of the dilated operator, into a hypoelliptic principal part, perturbed by a small operator. We will also give a brief introduction into Duhamel's formula in Subsection \ref{sec:duhamel}, which will be an important tool to study the perturbed operator. In Section \ref{sec:graded} we will introduce the coordinates that give the right dilations of the space. This coordinates are defined from a filtration of the tangent space to $x_0$ determined by the fields $f_0,f_1,\ldots, f_k$ and give rise to a graded structure around $x_0$, which defines an anisotropic dilation. In Section \ref{sec:4} we will define the nilpotent approximation, that determines the principal part of the vector fields $f_0,f_1,\ldots,f_k$, which define the principal operator \eqref{eq:operator princ}. We will compute also the integer $N$ appearing in the asymptotic \eqref{eq:order}, that derives from the change of the volume form under the dilations. In Section \ref{sec:5a} we will prove the asymptotic \eqref{eq:order}, by  using the tools introduced in the previous sections. In the following Section \ref{sec:5} we will focus our study on the operator derived from the nilpotent approximating system and its associated control problem. By proving a modification of Stroock and Varadhan's support theorem, we will give a necessary and sufficient condition for the positivity of the fundamental solution of the principal operator, that is based on the controllability of the approximating control system. Finally we will end the paper with Section \ref{sec:examples}, where we show a series of examples, illustrating how this formula recovers in particular the known results recalled in the introduction.

%%%%%%%%%%%%%%%%%%%%%%%%%%%%%%%%%%%%%%%%%%%%%%%%%%%%%%%%%%%%%%%%%%%%%%%%%%%%%%%%%%%%%%%%%%%%%%%%%%%%%%%%%%%%%%%%%%%%%%%%%%%%%%%%%%%%%%%%%%%%%%%%%%%%%%%%%%%%%%%%

\section{The fundamental solution and its behavior under the action of a dilation}\label{sec:2}
Let $M$ be an orientable $n$-dimensional manifold (without boundary) and let $\mu$ be a volume form on $M$. Given $f_0,f_1,\ldots,f_k$ smooth vector fields on $M$ we consider the following partial differential operator:
\begin{equation}
\frac{\partial \varphi}{\partial t}-f_0(\varphi)-\frac{1}{2}\sum_{i=1}^k f_i^2(\varphi) \qquad \forall \varphi\in C^\infty(\bbR\times M).
\label{eq:partial de}
\end{equation}

We will call $f_0$ the \emph{drift} field and we will denote by $L$ the operator $f_0+\frac{1}{2}\sum_{i=1}^k f_i^2$. Let us recall the definition of fundamental solution.

\begin{definition}
The \emph{fundamental solution} of an operator $\frac{\partial}{\partial t}-\calL$ over $\bbR\times M$ with respect to the volume $\mu$ is a function $p(t,x,y)\in C^\infty(\bbR^+\times M\times M)$ such that
\begin{itemize}
	\item[$\bullet$] for every fixed $y\in M$, it holds $\frac{\partial }{\partial t}p(t,x,y)=\calL_x p(t,x,y)$, where the operator $\calL$ acts on the $x$ variable;
	\item[$\bullet$] for any $\varphi_0\in C^\infty_0(M)$, we have 
	$$\lim_{t\searrow 0}\int_M p(t,x,y)\varphi_0(y) \mu(y)=\varphi_0(x) \qquad \mbox{ uniformly in }t.$$
\end{itemize}
In other words, if we want to solve the partial differential equation $\frac{\partial \varphi}{\partial t}=\calL \varphi$ with initial condition $\varphi(0,x)=\varphi_0(x)$, the fundamental solution allows to reconstruct $\varphi$ by convolution of $\varphi_0$ with $p(t,x,y)$.
\end{definition}

\begin{remark}\label{th:volume}
\emph{The fundamental solution depends on the given volume $\mu$ in the following way: let $\mu$ and $\nu$ be two volume forms on $M$, and let $g$ be a smooth function such that $\nu= e^g \mu$. Let $p_\mu$ and $p_\nu$ denote the fundamental solutions of $\frac{\partial}{\partial t}-\calL$ with respect to $\mu$ and $\nu$ respectively. Then for every initial condition $\varphi_0\in C_0^{\infty}(M)$, the solution $\varphi(t,x)$ of
\begin{equation}
\left\{ 
	\begin{array}{l}
		\frac{\partial \varphi}{\partial t}=f_0(\varphi)+\frac{1}{2}\sum_{i=1}^k f_i^2(\varphi)\\
		\varphi(0,x)=\varphi_0(x)
	\end{array}
\right.
\label{eq:1}
\end{equation}
is given by
\begin{equation*}
\begin{split}
\varphi(t,x)=&\int_{M}p_{\mu}(t,x,y)\varphi_0(y) \mu(y)\\
=&\int_{M}p_{\nu}(t,x,y)\varphi_0(y) \nu(y)=\int_{M}p_{\nu}(t,x,y)\varphi_0(y) e^{g(y)}\mu(y),
\end{split}
\end{equation*}
where the equalities follow since the solution is unique for smooth vector fields. Since $\varphi_0$ is arbitrary, we have
$$p_\mu(t,x,y)=e^{g(y)}p_\nu(t,x,y) \qquad \forall t>0, \forall x,y\in M.$$
}

\emph{From the point of view of the asymptotic of the fundamental solution on the diagonal, it follows that the two asymptotics are the same for both volume forms up to a multiplicative constant $e^{g(x_0)}\neq 0$ depending on the relation between the two volume forms and on the point where we compute the asymptotic. For the study of the small time asymptotic on the diagonal, we will then suppose $\mu=dx_1\wedge\ldots\wedge dx_n$ near the point $x_0$.}
\end{remark}

The existence of a fundamental solution for an operator $\frac{\partial}{\partial t}-\calL$ is not always guaranteed, but in 1967 H\"{o}rmander published an important paper on the properties of an operator like the one in \eqref{eq:partial de}. An operator $\calL$ is called \emph{hypoelliptic} if whenever, for a function $\varphi$, $\calL\varphi\in C^\infty(U)$ on an open set $U$, then $\varphi\in C^\infty(U)$. In \cite{article:Hormander} H\"{o}rmander proved a condition of hypoellipticity for operators of the form \eqref{eq:partial de}, namely if
$$\Lie_x\{\frac{\partial}{\partial t}-f_0,f_1,\ldots,f_k\}=\bbR\times T_xM \qquad \forall x\in M,$$
then the operator $\frac{\partial}{\partial t}-L$ is hypoelliptic. Throughout this paper we will always assume that the fields $f_0,f_1,\ldots,f_k$ satisfy H\"{o}rmander condition.

Under this assumption of hypoellipticity, the operator in \eqref{eq:partial de} admits a fundamental solution, $p(t,x,y)$ (see the backward Kolmogorov equation and \cite{book:IkedaWatanabe}, Chapter 4.6 for the study on $\bbR^n$ and Chapter 5 for the equation on a manifold). It is given by the probability density function of the process $\xi_t$ that satisfies the following stochastic differential equation in Stratonovich form
\begin{equation}
\left\{
\begin{array}{l}
d\xi_t=f_0(\xi_t)dt+\sum_{i=1}^k f_i(\xi_t) \circ dw_i,\\
\xi_0=x
\end{array}
\right.
\label{eq:differential equation}
\end{equation}
where $\xi_t$ is a process on $M$ and $(w_1,\ldots,w_k)$ denotes a $k$-dimensional Brownian motion. More explicitly, the fundamental solution of \eqref{eq:partial de} with respect to a volume form $\mu$ on $M$ is the function $p(t,x,y)$ such that for every measurable set $A\subset M$, the probability of the process $\xi_t$ to be in $A$ at time $t$ is given by
$$P[\xi_t\in A|\xi_0=x]=\int_A p(t,x,y)\mu(y).$$
Moreover, for every initial condition $\varphi_0(x)$, the corresponding solution of the differential equation is
$$\varphi(t,x)=\int_M p(t,x,y)\varphi_0(y)\mu(y)=\mathbb{E}[\left.\varphi_0(\xi_t)\right| \xi_0=x],$$
that is the expectation value of $\varphi_0(\xi_t)$ at time $t$, knowing that $\xi_0=x$.

Let $x_0$ be a fixed point of the drift field, $f_0(x_0)=0$, our goal in this paper is to understand the small time behaviour of the fundamental solution at the point $x_0$. Our method will proceed as follows: we will dilate the space around the point $x_0$, and we will rescale the time accordingly, so that the space around $x_0$ will be magnified and we can study better the behaviour of the solution at $x_0$.

\begin{definition}\label{th:dilation1}
Let $(U,x)$ be any coordinate neighborhood of $x_0$ and let $(w_1,\ldots,w_n)$ be positive integers, that we will call \emph{weights} of the coordinates $(x_1,\ldots,x_n)$.

For $\epsilon>0$ we define the \emph{dilation}, $\delta_\epsilon$, of order $\epsilon$ and weights $(w_1,\ldots,w_n)$ around $x_0$ as the function $\delta_\epsilon:U\rightarrow U$, such that
\begin{equation}
\delta(x_1,\ldots,x_n):=(\epsilon^{w_1}x_1,\epsilon^{w_2}x_2, \ldots, \epsilon^{w_n}x_n) \quad\qquad \forall x=(x_1,\ldots,x_n)\in U.
\label{eq:dilation1}
\end{equation}
For $\epsilon>1$ we understand the dilation defined actually on a smaller neighborhood of $x_0$ in $U$.
\end{definition}

Under the action of the dilations $\delta_\epsilon$, the coordinate functions and the coordinate vector fields behave as
\begin{equation}
x_i\circ \delta_\epsilon=\epsilon^{w_i} x_i, \qquad\qquad \delta_{\frac{1}{\epsilon}*}\frac{\partial}{\partial x_i}=\frac{1}{\epsilon^{w_i}}\frac{\partial}{\partial x_i}\qquad\qquad \forall 1\leq i\leq n.
\label{eq:b1}
\end{equation}
Let the volume $\mu$ be represented in the coordinate neighborhood $(U,x)$ by $\mu=dx_1\wedge\ldots\wedge dx_n$. By Remark \ref{th:volume} this assumption is not restrictive for the study of the asymptotic along the diagonal. Then the volume form $\mu$ changes under the action of the dilation $\delta_\epsilon$ as
\begin{equation}
\delta_{\epsilon}^*(dx_1\wedge\ldots\wedge dx_n)=\epsilon^{\sum_{i=1}^n w_i}dx_1\wedge\ldots\wedge dx_n.
\label{eq:volume transform}
\end{equation}

When we apply a dilation to the space around $x_0$ and we rescale the time variable, also the fundamental solution is changed accordingly, as it is proved in the following proposition.

\begin{proposition}\label{th:diffeo}
Let $(U,x)$ be a coordinate neighborhood around the point $x_0$ and let $\mu$ be a volume form on $M$ such that $\mu=dx_1\wedge\ldots\wedge dx_n$ in $U$. For weights $(w_1,\ldots,w_n)$ and $0<\epsilon<1$ consider the dilation $\delta_\epsilon:U\longrightarrow U$. Let $p(t,x,y)$ be the fundamental solution of the operator in \eqref{eq:partial de} with respect to the volume $\mu$. Then the fundamental solution on $U$ of the operator
\begin{equation}
\frac{\partial }{\partial t}-\epsilon^a\left(\delta_{1/\epsilon*}f_0+\frac{1}{2}\sum_{i=1}^k \left(\delta_{1/\epsilon*}f_i\right)^2\right)
\label{eq:modified}
\end{equation}
is the function
\begin{equation}
q_\epsilon(t,x,y):=\epsilon^{\sum_{i=1}^n w_i}\; p(\epsilon^a t, \delta_\epsilon(x),\delta_\epsilon(y)) \qquad \forall x,y\in U.
\end{equation}
Here $\delta_{1/\epsilon*}X$ denotes the pushforward of a vector field $X$ under the action of $\delta_\epsilon$ and $a$ is a real positive number.
\end{proposition}

\begin{remark}
\emph{The coefficient of normalization $\epsilon^{\sum_{i=1}^n w_i}$, that we have used to define $q_\epsilon$, is necessary in order to reconstruct all the solutions of the differential operator, by convolution with $q_\epsilon$. This coefficient appears as soon as we make a change of coordinates in the integral of the convolution. Moreover, we will see that from this coefficient we find the order of the asymptotic of the fundamental solution for small time.}
\end{remark}

\begin{proof}
First of all notice that the dilation $\delta_{1/\epsilon}:\delta_\epsilon(U)\rightarrow U$ can be defined only on the smaller neighborhood $\delta_\epsilon(U)$ of $U$. Then the fields $\delta_{1/\epsilon*}f_i$ are vector fields just on the coordinate neighborhood $U$.

Next let us prove the first property of the fundamental solution, i.e. that the function $q_\epsilon$ is a solution of the operator in \eqref{eq:modified}. For convenience, we call $\psi$ the dilation from $\bbR^+\times \delta_{1/\epsilon}U$ to $\bbR^+\times U$ defined by $\psi(t,x):=(\epsilon^a t, \delta_\epsilon(x))$. Then the function $q_\epsilon$ can be written as $q_\epsilon(t,x,y)=p(\psi(t,x),\delta_\epsilon(y))$ and the operator in \eqref{eq:modified} is
$$\psi^{-1}_*\frac{\partial }{\partial t}-\psi^{-1}_*f_0-\frac{1}{2}\sum_{i=1}^k \left(\psi^{-1}_*f_i\right)^2,$$
where we are using a little abuse of notation, by considering $\frac{\partial}{\partial t}, f_0,\ldots,f_k$ as vector fields defined on the product space $\bbR^+\times U$. Recall the definition of pushforward of a vector field $X$ under the action of a diffeomorphism $\varphi$: for every function $g$ we have $\left.\varphi_{x*}(X)(f)\right|_{\varphi(x)}=X_x(f\circ\left.\varphi\right|_x)$. Then we compute for fixed $y\in U$
\begin{eqnarray*}
\left(\psi^{-1}_*\frac{\partial }{\partial t}\right)q_\epsilon(t,x,y)&=&\epsilon^{\sum_{i=1}^n w_i}\;\left.\left(\psi^{-1}_*\frac{\partial }{\partial t}\right)\right|_{(t,x,y)}(p(\psi(t,x),\delta_{\epsilon}y))\\
&=&\epsilon^{\sum_{i=1}^n w_i}\;\psi_* \left(\left.\left(\psi^{-1}_*\frac{\partial }{\partial t}\right)\right|_{(t,x,y)}\right)\left.p\right|_{(\psi(t,x),\delta_{\epsilon}y)}\\
&=&\epsilon^{\sum_{i=1}^n w_i}\;\left.\frac{\partial }{\partial t}\right|_{(\psi(t,x),\delta_{\epsilon}y)}\left.p\right|_{(\psi(t,x),\delta_{\epsilon}y)}\\
&=&\epsilon^{\sum_{i=1}^n w_i}\; \left. L\right|_{(\psi(t,x),y)}\;\left.p\right|_{(\psi(t,x),\delta_{\epsilon}y)}
\end{eqnarray*} 
where the last equality follows since $p$ is the fundamental solution of the operator $\frac{\partial}{\partial t}-L$. Applying the same computations to the vector fields $\psi^{-1}_* f_i$ for $i=0,\ldots,k$, we find that $q_\epsilon$ satisfies
$$\psi^{-1}_*\frac{\partial }{\partial t}q_\epsilon(t,x,y)=\psi^{-1}_*f_0(q_\epsilon(t,x,y))+\frac{1}{2}\sum_{i=1}^k \left(\psi^{-1}_*f_i\right)^2(q_\epsilon(t,x,y))$$
and hence $q_\epsilon$ is a solution for the operator in \eqref{eq:modified}.

Let us prove the second property of a fundamental solution. Here we will see that the constant of normalization $\epsilon^{\sum_{i=1}^n w_i}$ in the definition of $q_\epsilon$ is exactly the parameter that we need in order to construct the other solutions of the partial differential equation by convolution with the fundamental solution. Indeed let us prove that for any $\varphi_0\in C^\infty_0(U)$, it holds
$$\lim_{t\searrow 0}\int_{U}q_\epsilon(t,x,y)\varphi_0(y)\mu(y)=\varphi_0(x).$$
This follows by a change of variable and the same property valid for the fundamental solution $p(t,x,y)$:
\begin{eqnarray*}
\lim_{t\searrow 0}\int_{U}q_\epsilon(t,x,y)\varphi_0(y)\mu(y)&=&\lim_{t\searrow 0}\int_{U}\epsilon^{\sum_{i=1}^n w_i} p(\epsilon^a t, \delta_\epsilon x,\delta_\epsilon y)\varphi_0(y)\mu(y)\\
&=&\lim_{t\searrow 0}\int_{M}\epsilon^{\sum_{i=1}^n w_i} p(\epsilon^a t, \delta_\epsilon x,\delta_\epsilon y)\varphi_0(y)\mu(y)
\end{eqnarray*}
We can integrate on $M$, because, since $\varphi_0$ has compact support in $U$, we can extend $\varphi_0$ to be zero outside $U$. Now let us do a change of variable with $z=\delta_\epsilon y$. As computed in \eqref{eq:volume transform}, the volume is transformed as $\mu(z)=\epsilon^{\sum_{i=1}^n w_i}\mu(y)$. Then
\begin{eqnarray*}
\lim_{t\searrow 0}\int_{U}q_\epsilon(t,x,y)\varphi_0(y)\mu(y)&=&\lim_{t\searrow 0}\int_{M}\epsilon^{\sum_{i=1}^n w_i} p(\epsilon^a t, \delta_\epsilon x,z)(\varphi_0\circ \delta_{1/\epsilon})(z)\frac{\mu(z)}{\epsilon^{\sum_{i=1}^n w_i}}\\
&=&(\varphi_0\circ \delta_{1/\epsilon})(\delta_\epsilon x)=\varphi_0(x)
\end{eqnarray*}
where the second equality follows because $p$ is a fundamental solution.
\end{proof}

Let us investigate better the behaviour of the fields $f_i$ under the action of the dilations. Let us write every component $f^{(j)}_i$ of the fields $f_i$ in a Taylor expansion centered in $x_0=0$ for $x$ in the coordinate neighborhood $U$
$$f_i^{(j)}(x)\frac{\partial}{\partial x_j}=f_i^{(j)}(0)\frac{\partial}{\partial x_j}+\sum_{l=1}^n \frac{\partial f_i^{(j)}(0)}{\partial x_l} x_l \frac{\partial}{\partial x_j}+o(x). $$
By the properties of a dilation acting on the coordinate functions and on the coordinate vector fields, \eqref{eq:b1}, when we apply a dilation to the vector fields $f_i$, every component has a different degree with respect to $\epsilon$. Depending on the value of the weights $(w_1,\ldots,w_n)$ and on the coefficients of the Taylor expansion of the fields $f_i$, for every $i=0,\ldots, k$, there exist an integer $\alpha_i$ and a principal vector field $\hat{f}_i$ such that
$$\delta_{\frac{1}{\epsilon}*}f_i=\frac{1}{\epsilon^{\alpha_i}}\hat{f}_i+o\left(\frac{1}{\epsilon^{\alpha_i}}\right),$$ 
where $\hat{f}_i$ contains the components of every $f_i^{(j)}\frac{\partial}{\partial x_j}$ that is homogeneous of degree $-\alpha_i$ with respect to the dilations. Applying this formula to the dilated operator in \eqref{eq:modified}, we find that the operator $L$ rescales as 
$$\delta_{1/\epsilon*}f_0+\frac{1}{2}\sum_{i=1}^k \left(\delta_{1/\epsilon*}f_i\right)^2=\frac{1}{\epsilon^{\alpha_0}}\hat{f}_0+\frac{1}{2}\sum_{i=1}^k\frac{1}{\epsilon^{2\alpha_i}}\hat{f}_i^2+o\left(\frac{1}{\epsilon^{\alpha}}\right)$$
where $\alpha:=\max\{\alpha_0,2\alpha_1,\ldots,2\alpha_k\}$.

The main task in our study will be to find suitable coordinates and good weights $w_i$, so that all the principal parts of the vector field $f_0$ and of $f_1^2,\ldots,f_k^2$ rescale with the same degree under the dilations. In this way the principal part in the dilated operator is homogeneous of order $-\alpha$ and if we choose $a=\alpha$, the operator $L$ in \eqref{eq:modified} can be written as
\begin{equation}
\epsilon^\alpha\left(\delta_{1/\epsilon*}f_0+\frac{1}{2}\sum_{i=1}^k \left(\delta_{1/\epsilon*}f_i\right)^2\right)=\hat{f}_0+\frac{1}{2}\sum_{i=1}^k\hat{f}_i^2+\calX_\epsilon.
\label{eq:a1}
\end{equation}
where $\calX_\epsilon$ is a differential operator that goes to $0$ as $\epsilon$ goes to zero. As proved in Proposition \ref{th:diffeo}, its fundamental solution on $U$ is given by 
\begin{equation}
q_\epsilon(t,x,y):=\epsilon^{\sum_{i=1}^n w_i} p(\epsilon^\alpha t,\delta_\epsilon x, \delta_\epsilon y).
\label{eq:qepsilon}
\end{equation}

To an operator like the one in \eqref{eq:a1} we can apply Duhamel's formula, that gives the asymptotic of the fundamental solution as a perturbation of the asymptotic of the fundamental solution of the principal operator.

\subsection{Duhamel's formula}\label{sec:duhamel}
In this section we recall briefly a famous formula, called Duhamel's formula, which allows to find the asymptotic of the fundamental solution of a perturbed operator, once we have the explicit fundamental solution of the principal part of the operator. This method is presented, among others, in Chapter 3 of \cite{book:RosenbergLaplacian} and in \cite{article:Barilari}.

Let $\calL$ be an operator on a Hilbert space with fundamental solution $p(t,x,y)$ and let us define the following operator on the Hilbert space
\begin{equation*}
e^{t\calL}\varphi(x)=\int p(t,x,y)\varphi(y)dy.
\end{equation*} 
By the properties of the fundamental solution this is an heat operator  $e^{t\calL}$, i.e. an operator such that
\begin{equation*}
\frac{\partial e^{t\calL}\varphi}{\partial t}=\calL e^{t\calL}\varphi \quad \mbox{ and } \quad \lim_{t\rightarrow 0}e^{t\calL}\varphi=\varphi,
\end{equation*}

Suppose that $\calL$ can be decomposed in a sum, 
\begin{equation*}
\calL=\calL_0+\calX,
\end{equation*} 
of a principal part, $\calL_0$, and a perturbation, $\calX$, and assume that $\calL_0$ has a well defined heat operator $e^{t\calL_0}$. Then Duhamel's formula allows to reconstruct the heat operator of $\calL$ by a perturbation of the heat operator of $\calL_0$ (see Chapter 3 of \cite{book:RosenbergLaplacian} for a proof), namely
\begin{equation}
e^{t\calL}=e^{t\calL_0}+\int_0^t e^{(t-s)\calL}\calX e^{s\calL_0}ds=e^{t\calL_0}+e^{t\calL}*\calX e^{t\calL_0}
\label{eq:duhamel2}
\end{equation}
where with $*$ we denote the convolution operator between two operators, $A(t)$ and $B(t)$, on the Hilbert space:
$$(A*B)(t)=\int_0^tA(t-s)B(s)ds.$$ 
Let $a(t,x,y)$ and $b(t,x,y)$ be the heat kernels of $A(t)$ and $B(t)$ respectively and let $X$ be an operator. Then the kernel of $(A*XB)(t)$ is obtained as follows: let $\varphi$ be a function in the Hilbert space, then
\begin{equation*}
\begin{split}
\left[(A*XB)(t)\varphi\right](x)&=\left[\int_0^tA(t-s)XB(s)ds\;\varphi\right](x)\\
&=\left[\int_0^tA(t-s)\left[X\int_M b(s,\cdot,y)\varphi(y)dy\right]ds\right](x)\\
&=\int_0^t\int_M a(t-s,x,z)X_z\left(\int_M b(s,z,y)\varphi(y)dy\right)dzds\\
&=\int_M\left(\int_0^t\int_M a(t-s,x,z)X_zb(s,z,y))dzds\right)\varphi(y)dy
\end{split}
\end{equation*}
so the heat kernel of $(A*XB)(t)$ is
\begin{equation*}
(a*Xb)(t,x,y)=\int_0^t\int_M a(s,x,z)X_zb(t-s,z,y)dzds.
\end{equation*}
From \eqref{eq:duhamel2} we can now derive an approximation of the heat kernel $p(t,x,y)$ of the perturbed operator $\calL$, by means of the heat kernel $p_0(t,x,y)$ of the principal operator:
\begin{equation}
p(t,x,y)=p_0(t,x,y)+(p*\calX p_0)(t,x,y).
\label{eq:duhamel asymptotic}
\end{equation}
We will make use of this formula in the next section to find the asymptotic expansion of the fundamental solution of \eqref{eq:operator}.

\subsection{The perturbative method}
We can apply Duhamel's formula to the operator in \eqref{eq:a1}. Indeed it is the sum of a principal operator
$$L_0=\hat{f}_0+\frac{1}{2}\sum_{i=1}^k\hat{f}_i^2$$
perturbed by a small operator $\calX_\epsilon$. If we find good coordinates and weights, so that $\frac{\partial}{\partial t}-L_0$ admits a fundamental solution $q_0(t,x,y)$, then by Duhamel's formula \eqref{eq:duhamel asymptotic} the asymptotic of the fundamental solution $q_\epsilon$ is
\begin{equation}
q_\epsilon(t,x,y)=q_0(t,x,y)+(q_\epsilon*\calX_\epsilon q_0)(t,x,y).
\label{eq:qepsilon1}
\end{equation}
From here we will immediately conclude our study, indeed let us choose in the definition of $q_\epsilon$, \eqref{eq:qepsilon}, $x=x_0$, let $\epsilon$ go to zero as $\epsilon=t^{1/\alpha}$ and fix the time variable $t=1$, then we get
$$p(t,x_0,x_0)=\frac{q_\epsilon(1,x_0,x_0)}{\epsilon^{\sum_{i=1}^n w_i}}=\frac{1}{\epsilon^{\sum_{i=1}^n w_i}}\left(q_0(1,x_0,x_0)+(q_\epsilon*\calX_\epsilon q_0)(1,x_0,x_0)\right).$$
In conclusion, in the choice of the coordinates $(U,x)$ and the weights $(w_1,\ldots,w_n)$ it will be important that the principal parts, $\hat{f}_0,\hat{f}_1,\ldots,\hat{f}_k$, of the vector fields make homogeneous the principal part of the dilated operator and, moreover, that they satisfy H\"{o}rmander condition, so that it is guaranteed the existence of a fundamental solution $q_0$ of the principal operator $L_0$. Finally we will need to check that the remainder term $(q_\epsilon*\calX_\epsilon q_0)(1,x_0,x_0)$ in the asymptotic of $p$ goes to zero, as $t$ (and therefore $\epsilon$) goes to zero.

%%%%%%%%%%%%%%%%%%%%%%%%%%%%%%%%%%%%%%%%%%%%%%%%%%%%%%%%%%%%%%%%%%%%%%%%%%%%%%%%%%%%%%%%%%%%%%%%%%%%%%%%%%%%%%%%%%%%%%%%%%%%%%%%%%%%%%%%%%%%%%%%%%%%%%%%%%%%%%%%%%%%%

\section{Graded structure induced by a filtration}\label{sec:graded}
In this section we will introduce some notation and recall the definition of local graded structure of a manifold, induced by a filtration. This terminology will be essential to find the right coordinate to rescale the differential operator $L$ and to compute the order of the asymptotic of the fundamental solution. We will constantly refer to Bianchini and Stefani's paper \cite{art:BianchiniStefani}.

\subsection{Chart adapted to a filtration}
Let $M$ be a $n$ dimensional smooth manifold and $f_0,f_1,\ldots,f_k$ be smooth vector fields on $M$, that satisfy H\"{o}rmander condition \eqref{eq:hormander}. We will consider the hypoelliptic operator on $\bbR\times M$ defined as
$$\frac{\partial}{\partial t}-f_0-\sum_{i=1}^kf_i^2.$$
As noticed in the previous section the role played by the drift field $f_0$ and the other vector fields, $\{f_1,\ldots,f_k\}$, in the sum of squares, is different, and in particular the fields $\{f_1,\ldots,f_k\}$ are applied twice as many times as the drift field is. For this reason we want to treat differently the two kinds of fields, by giving to them two different \emph{weights}. 

Let $\Lie X$ be the Lie algebra on $\bbR$ generated by a set $\{X_0,X_1,\ldots,X_k\}$ of noncommutative indeterminates. 

\begin{definition}
For every bracket $\Lambda$ in $\Lie X$  we denote by $|\Lambda|_i$ the number of times that the indeterminate $X_i$ appears in the definition of $\Lambda$. We will call this number the \emph{length} of $\Lambda$ with respect to $X_i$. 
\end{definition}
For example, the bracket $\Lambda=[X_0,[X_2,X_0]]$ has lengths $|\Lambda|_0=2$, $|\Lambda|_1=0$ and $|\Lambda|_2=1$, and it has zero length with respect to any other indeterminate. 

By fixing a weight, $l_i$, to every indeterminate $X_0,\ldots,X_k$ be can define the weight of a bracket $\Lambda$.

\begin{definition}\label{th:weight}
Given a set of integers $(l_0,l_1,\ldots,l_k)$, we define the \emph{weight} of a bracket $\Lambda\in\Lie X$ as
%\begin{equation}
$$||\Lambda||:=\sum_{i=0}^k l_i |\Lambda|_i \qquad \mbox{ if } \Lambda\neq 0$$
%\label{eq:weight}
%\end{equation}
and we set $||0||=0$.
\end{definition}

In order to give different importance to the drift field, with respect to the other vector fields, in the following we will fix the integers to be 
\begin{equation}
l_0=2 \quad \mbox{ and } \quad l_1=\ldots=l_k=1.
\label{eq:weights}
\end{equation}
This means that the indeterminate $X_0$ will have weight $2$, while the other indeterminates will have weight $1$. For more complex Lie brackets, we have for example that the weight of the bracket considered before $\Lambda=[X_0,[X_2,X_0]]$ is 
$$||\Lambda||=2\cdot |\Lambda|_0+1\cdot |\Lambda|_2=5.$$

By means of the weight of the indeterminates we introduce now a filtration of the Lie algebra spanned by $f_0,f_1,\ldots,f_k$ in the following way. For every bracket $\Lambda$ in $\Lie X$ we denote by $\Lambda_f$ the vector field on $M$ obtained by replacing every indeterminate $X_i$ with the corresponding field $f_i$ for $0\leq i\leq k$. Then we define an increasing filtration $\mathcal{L}=\{L_i\}_{i\geq 0}$ of $\mathrm{Vec}(M)$ by
\begin{equation}
L_i=\mathrm{span}\{\Lambda_f:\Lambda\in \mathrm{Lie} X, ||\Lambda||\leq i\}.
\label{eq:filtration}
\end{equation}
In other words $L_i$ is the subalgebra of $\mathrm{Vec}(M)$ that contains all the vector fields obtained from a bracket of weight less then or equal to $i$. In particular, following our choice of weights, the first subspaces of the filtration are
\begin{eqnarray*}
L_0&=&\{0\}\\
L_1&=&\mathrm{span}\{f_1,\ldots,f_k\}\\
L_2&=&\mathrm{span}\{f_i,[f_i,f_j],f_0 : i,j=1,\ldots, k \}\\
L_3&=&\mathrm{span}\{f_i,[f_i,f_j],f_0,[f_i,[f_j,f_h]],[f_0,f_i] : i,j,h=1,\ldots, k \}\\
&\vdots &
\end{eqnarray*}
Notice moreover that for every $i,j\geq 0$, the following properties hold
\begin{itemize}
	\item[$\bullet$] $L_i\subset L_{i+1}$
	\item[$\bullet$] $[L_i,L_j]\subset L_{i+j}$
	\item[$\bullet$] $\bigcup_{i\geq 0} L_i=\Lie\{f_0,f_1,\ldots,f_k\}$ and $\bigcup_{i\geq 0} L_i(x)=T_xM$, for every $x\in M$, since by assumption the family $\{f_0,\ldots,f_k\}$ satisfies the weak H\"{o}rmander condition \eqref{eq:hormander}.
\end{itemize}
When we evaluate $\mathcal{L}$ at the stationary point $x_0$ we get a stratification of the tangent space $T_{x_0}M$ at $x_0$ and this stratification will induce a very peculiar choice of coordinates around $x_0$. 

Let 
$$d_i:=\dim L_i(x_0) \qquad\quad \forall i\geq 0.$$ In particular, $d_0=0$ and $d_1\leq k$. Moreover, by H\"{o}rmander condition \eqref{eq:hormander}, there exists a smallest integer $m$ such that $L_m(x_0)=T_{x_0}M$. We will call this number the \emph{step} of the filtration $\mathcal{L}$ at $x_0$. 

The filtration $\calL$ induces a particular choice of coordinates centered at $x_0$, as proved by the following proposition.

\begin{proposition}[Bianchini, Stefani \cite{art:BianchiniStefani}]\label{th:adapted chart}
There exists a chart $(U,x)$ centered at $x_0$ such that for every $1\leq j\leq m$
\begin{itemize}
	\item[(i)] $L_j(x_0)=\mathrm{span}\{\frac{\partial}{\partial x_1},\ldots,\frac{\partial}{\partial x_{d_j}}\}$
	\item[(ii)] $D x_h(x_0)=0$ for every differential operator $D\in\calA^j:=\{ Z_1\cdots Z_l$ with $Z_s\in L_{i_s}$ and $i_1+\cdots+i_l\leq j\}$ and for every $h>d_j$.
\end{itemize}
\end{proposition}

\begin{definition}
We will call a chart that satisfies the properties of Proposition \ref{th:adapted chart} an \emph{adapted chart to the filtration $\mathcal{L}$ at $x_0$}.
\end{definition}

Since this kind of coordinates will reveal to be very important in our study, we give here the proof of the proposition, which relies upon the following Lemma:
\begin{lemma}\label{th:adapted chart lemma}
Let $m$ be the step of the filtration $\mathcal{L}$ at $x_0$ and let $j<m$ be an integer. Let $\varphi\in C^\infty(M)$ be such that $d_{x_0}\varphi\neq 0$ and $Z\cdot \varphi(x_0)=0$ for all $Z\in L_{j}$. Then there exists an open neighborhood $U$ of $x_0$ and a function $\hat{\varphi}\in C^\infty(U)$ such that 
$$d_{x_0}\varphi=d_{x_0}\hat{\varphi} \qquad \mbox{and}\qquad D\cdot \hat{\varphi}(x_0)=0 \quad \forall D\in\calA^j=\{ Z_1\cdots Z_l : \; Z_s\in L_{i_s},i_1+\cdots+i_l\leq j\}.$$
\end{lemma}
\begin{proof}
Let $Y_1,\ldots,Y_n$ be vector fields on $M$ such that they form a basis of $T_{x_0}M$ at $x_0$ and such that
\begin{itemize}
	\item[$\bullet$] $\{Y_1,\ldots,Y_{d_i}\}$ are in $L_i$ and form a basis of $L_i$ at $x_0$, for every $i\leq j$,
	\item[$\bullet$] $Y_i\cdot \varphi(x_0)=0$ for every $i\leq n-1$,
	\item[$\bullet$] $Y_n\cdot \varphi(x_0)=1$.
\end{itemize}
We choose the chart $y=(y_1,\ldots,y_n)$ as the local inverse of the map
$$(y_1,\ldots,y_n)\mapsto e^{y_nY_n}\cdots e^{y_1Y_1}x_0.$$
Then the function $\hat{\varphi}:= y_n$ is such that $d_{x_0}\varphi=d_{x_0}\hat{\varphi}$.

Let $D= Z_1\cdots Z_l\in \calA^j$ with $Z_s\in L_{i_s}$ and $i_1+\cdots+i_l\leq j$ and let us prove the second property required for the function $\hat{\varphi}$, by induction on $l$. Since $d_{x_0}\varphi=d_{x_0}\hat{\varphi}$ and $Z\cdot \varphi(x_0)=0$ for all $Z\in L_{j}$ by hypothesis, the property is satisfied for $l=1$.
For $l>1$, since $Z_l\in L_{i_l}$ we can write $Z_{l}(x_0)=\sum_{i=1}^{d_{i_l}}a_iY_i(x_0)$ for some $a_i$ so
$$D\cdot \hat{\varphi}(x_0)=\sum_{i=1}^{d_{i_l}}a_i\left(Z_{l-1}\cdot Y_i\cdot Z_{l-2}\cdots Z_1+[Y_i,Z_{l-1}]\cdot Z_{l-2}\cdots Z_1\right)\cdot\hat{\varphi}(x_0)$$
The second component on the left side vanishes because, by the definition of the filtration, $[L_i,L_h]\in{L_{i+h}}$, so we can apply on this component the induction hypothesis. By applying again the same commutation we have
$$D\cdot \hat{\varphi}(x_0)=\sum_{i=1}^{d_{i_l}}a_i Z_{l-1}\cdots Z_1\cdot Y_i\cdot\hat{\varphi}(x_0).$$
Iterating the same procedure also to $Z_{l-1},\ldots,Z_1$ we can write $D\cdot \hat{\varphi}(x_0)$ as a linear combination of elements of the type
$$Y_{i_l}\cdots Y_{i_1}\cdot \hat{\varphi}(x_0),$$
with $1\leq i_h\leq d_j<n$ for every $h=1,\ldots, l$. Therefore we get $D\cdot \hat{\varphi}(x_0)=D\cdot y_n(x_0)=0.$
\end{proof}

\begin{proof}[Proof of Proposition \ref{th:adapted chart}]
Let $(U,x)$ be any chart centered at $x_0$. We can get a chart with property $(i)$ of the proposition by a linear change of coordinates. Let us still denote it by $(U,x)$. For every $i\leq n$ let $j$ be such that $d_{j}<i\leq d_{j+1}$. Then the coordinate function $x_i$ satisfies the hypothesis of the Lemma with respect to the integer $j$. By applying the Lemma to each function of the chart we get the statement.
\end{proof}

We will see an example of an adapted chart in Section \ref{sec:nilpotent approximation}.

\subsection{Graded structure}
For $1\leq i\leq m$ we define the integers 
$$k_i:= d_i-d_{i-1},$$
which indicate the number of new coordinates achieved with every new layer of the filtration at $x_0$.

\begin{definition}\label{th:dilation}
Let us denote a point $x=(x_1,\ldots,x_n)\in U$ by the $m$-tuple $(x^1,x^2,\ldots,x^m)\in \bbR^{k_1}\oplus\bbR^{k_2}\oplus\ldots\oplus\bbR^{k_m}$, where each component $x^i:=(x_{d_{i-1}+1},\ldots,x_{d_i})$ is a vector of length $k_i$. Then for every $\epsilon>0$ we define the \emph{anisotropic dilations} around $x_0$ of factor $\epsilon$ as
\begin{equation*}
\delta_\epsilon(x)=\delta_\epsilon(x^1,\ldots,x^m):=(\epsilon x^1,\epsilon^2 x^2,\ldots, \epsilon^m x^m).
\end{equation*}
Notice that if $\epsilon>1$, the map $\delta_\epsilon$ is actually defined only in a smaller neighborhood of $x_0$, but we will still write the map $\delta_\epsilon:U\rightarrow U$, understanding "locally defined".
\end{definition}

The dilations $\delta_\epsilon$ act on every coordinate function and on the coordinate vector fields with a different weight, namely
\begin{equation*}
x_j\circ\delta_\epsilon=\epsilon^i x_j, \quad (\delta_\epsilon)_*\frac{\partial}{\partial x_j}=\epsilon^i\frac{\partial}{\partial x_j} \qquad \forall d_{i-1}<j\leq d_i.
\end{equation*}
For every $1\leq i\leq n$ let $w_j$ be the order of expansion of the coordinate $x_j$, that is $w_j:=i$ if $d_{i-1}<j\leq d_i$. We call $w_j$ the \emph{weight} of the coordinate $x_j$. Then the dilation $\delta_\epsilon$ is a particular choice of the dilations defined in Definition \ref{th:dilation1} with respect to the coordinates induced by the filtration $\calL$ and the weights $(w_1,\ldots,w_n)$.

Accordingly we define the \emph{weight} of a monomial to be
$$\mathcal{W}(x_{1}^{\alpha_1}\cdots x_{n}^{\alpha_n}):=\sum_{j=1}^n\alpha_j w_{j}$$
and the weight of a polynomial to be the greatest order of its monomials. Moreover we define the \emph{graded order}, $\mathcal{O}(g)$,  of a function $g\in C^\infty(U)$ to be the smallest weight of the monomials that appear in any Taylor approximation of $g$ at $x_0$.

For example, let $n=2$ and suppose that $x_1$ has weight 1 and $x_2$ has weight 2. Then the polynomial $x_1x_2-\frac{(x_1)^2(x_2)^2}{6}$ has weight 6, because the two monomials composing it are $x_1x_2$ of weight 3 and the rest of weight 6. On the other hand, $\sin(x_1x_2)=x_1x_2-\frac{(x_1)^2(x_2)^2}{6}+o\left((x_1)^2(x_2)^2\right)$ has graded order 3.

We extend these definitions to differential operators. We say that a polynomial vector field $Z$ is homogeneous of \emph{weight} $i$ if
$$\mathcal{W}(Z \varphi)=\mathcal{W}(\varphi)-i \qquad \forall \mbox{ monomial } \varphi \mbox{ of weight } \mathcal{W}(\varphi).$$
In other words $Z$ subtracts weight $i$ to every function. Then the weight of a polynomial vector field is the smallest weight of its homogeneous components. We define the \emph{graded order}, $\mathcal{O}(D)$, of a differential operator $D$ by saying that
$$\mathcal{O}(D)\leq j \quad \mbox{if and only if}\quad \mathcal{O}(D\cdot \varphi)\geq \mathcal{O}(\varphi)-j \qquad \forall \mbox{ polynomial } \varphi,$$
that is $D$ subtracts at most weight $j$ from the functions. For example the graded order of a vector field like $(x_{1}^{\alpha_1}\cdots x_{n}^{\alpha_n})\frac{\partial}{\partial x_h}$ is
$$\mathcal{O}\left((x_{1}^{\alpha_1}\cdots x_{n}^{\alpha_n})\frac{\partial}{\partial x_h}\right)=w_h-\left(\sum_{j=1}^n\alpha_j w_{j}\right)$$
Coming back to the previous example, the graded order of a field like $\sin(x_1x_2)\frac{\partial}{\partial x_1}$ is obtained as $1-\calO(\sin(x_1x_2))=-2$.

By means of the graded order we can give a generalization of the concept of Taylor approximation of a function up to weight $h$. Namely, for any $\varphi\in C^\infty(U)$ and every integer $h\geq0$, there is a unique polynomial $\varphi_{(h)}$ of weight $h$ such that $\mathcal{O}(\varphi-\varphi_{(h)})\geq h$.
\begin{definition}
The polynomial $\varphi_{(h)}$ is called the \emph{graded approximation of weight $h$ of $\varphi$} and it is the sum of the polynomials of weight less then or equal to $h$ in the formal Taylor expansion of $\varphi$ at $x_0$.
\end{definition}
For every vector field $V\in \mathrm{Vec}(U)$ and each integer $h\leq m$ there is a polynomial vector field $V_{(h)}$ of weight $h$ such that $\mathcal{O}(V-V_{(h)})\leq h-1$.
\begin{definition}
$V_{(h)}$ is called the \emph{graded approximation of weight $h$ of $V$} and it is the sum of the homogeneous vector fields of weight greater than or equal to $h$ in the formal Taylor expansion of $V$ at $x_0$.
\end{definition}
Notice that, since $V_{(h)}$ is a polynomial vector field, we can consider it as defined on the whole Euclidean space $\bbR^n$.

We will see in the next sections how to apply this graded structure in order to underline the most important properties of the operator in \eqref{eq:operator}, concerning the small time asymptotic of its fundamental solution.

%%%%%%%%%%%%%%%%%%%%%%%%%%%%%%%%%%%%%%%%%%%%%%%%%%%%%%%%%%%%%%%%%%%%%%%%%%%%%%%%%%%%%%%%%%%%%%%%%%%%%%%%%%%%%%%%%%%%%%%%%%%%%%%%%%%%%%%%%%%%%%%%%%%%%%%%%%%%%%%%%%%%%%

\section{Nilpotent approximation and the order of the dilations}\label{sec:4}
In this section we apply the graded structure, that we have just developed, to define a special class of vector fields, which approximate the original one $f_0,f_1,
\ldots, f_n$ and we will give an example to clarify the setting. Finally we will compute how the dilations change the volume form and we will introduce the order that will appear in the asymptotic of the heat kernel.

\subsection{Nilpotent approximation}\label{sec:nilpotent approximation}
Let $(x,U,w)$ be the graded structure around $x_0$ introduced in Section \ref{sec:graded} and let $f_0,f_1,\ldots,f_k$ be the vector fields used to define the filtration of $T_{x_0}M$. Then as proved in \cite{art:BianchiniStefani} Theorem 3.1, for every $f\in L_i$, we have  a bound on the graded order, namely $\calO(f)\leq i$, where $\calO$ is the graded order associated to the graded structure $(x,U,w)$ defined in Section \ref{sec:graded}. %Moreover $l_i=O(f_i)$ for all $i=0,\ldots, k$. 

Recall the integers $l_0,l_1,\ldots,l_k$ introduced in Section \ref{sec:graded} to define the filtration and denote by $\hat{f}_i$ the graded approximation of weight $l_i$ of $f_i$. In other words,  $\hat{f}_i$ has weight $l_i$ and $\calO(f_i-\hat{f}_i)< l_i$. 
\begin{definition}\label{th:nilpotent approximation}
The fields $\hat{f}_0,\hat{f}_1,\ldots,\hat{f}_k$ are called the \emph{nilpotent approximation} of $f_0,f_1,\ldots,f_k$.
\end{definition}
The fields $\hat{f}_i$, for $0\leq i\leq k$, are polynomials, so they can be defined on $\bbR^n$.

We can describe more precisely the structure of the approximating fields $\hat{f}_i$:
\begin{itemize}
	\item[$\bullet$] $\hat{f}_0$ contains the terms of weight $2$; therefore every component $\hat{f}_0^j$ of $\hat{f}_0$ depends only linearly on the coordinates of weight $w_j-2$ and more then linearly on the coordinates of less weight, but does not depend on the coordinates of weight greater then or equal to $w_j-1$, that are $x_h$ with $h>d_{j-2}$. Moreover, since $f_0$ vanishes in $x_0$, there are no constant terms.
	\item[$\bullet$] $\hat{f}_i$ contains the terms of weight $1$ for $i=1,\ldots,k$; therefore every component $\hat{f}_i^j$ of $\hat{f}_i$ depends only linearly on the coordinates of weight $w_j-1$ and more then linearly on the coordinates of less weight, but do not depend on the coordinates of weight greater then or equal to $w_j$, that are $x_h$ with $h>d_{j-1}$.
\end{itemize} 

To make the construction more clear we end this subsection with an example, in which we will present the filtration in $x_0$, the induced adapted chart and the graded structure, and we will find the related nilpotent approximation.

\begin{example}\label{example1}
\emph{Let $M=\bbR^2$, and let the number of controlled vector fields be $k=1$. Define the vector fields
\begin{equation*}
f_1:=\frac{\partial}{\partial x_1}+x_1\frac{\partial}{\partial x_2} \quad \mbox{ and } \quad f_0:=\sin(x_1^2)\frac{\partial}{\partial x_2}
\end{equation*}
and recall the choice of weights $l_0=2$ and $l_1=1$. The non vanishing Lie brackets that contribute to span the tangent space in any point are given by
$$[f_1,f_0]=2x_1 \cos(x_1^2)\frac{\partial}{\partial x_2}, \qquad [f_1,[f_1,f_0]]=(2 \cos(x_1^2)-4x_1^2\sin(x_1^2))\frac{\partial}{\partial x_2}.$$
Then H\"{o}rmander assumption \eqref{eq:hormander} holds in any point and the filtration defined in \eqref{eq:filtration} is equal to
\begin{itemize}
	\item[$\bullet$] $L_1=\mathrm{span}\{f_1\}$
	\item[$\bullet$] $L_2=\mathrm{span}\{f_1,f_0\}$
	\item[$\bullet$] $L_3=\mathrm{span}\{f_1,f_0,[f_1,f_0]\}$
	\item[$\bullet$] $L_4=\mathrm{span}\{f_1,f_0,[f_1,f_0],[f_1,[f_1,f_0]]\}$
\end{itemize}
In particular the filtration in $x_0$ is given by 
\begin{equation}
L_1(x_0)=L_2(x_0)=L_3(x_0)=\mathrm{span}\{\frac{\partial}{\partial x_1}\} \quad \mbox{ and }\quad L_4(x_0)=\bbR^2
\label{eq:example3}
\end{equation}
and the dimensions are: $d_1=d_2=d_3=1$, $d_4=2$.}

\emph{Let us find an adapted chart to the filtration at $x_0$. As one can easily see, the coordinates $(x_1,x_2)$ are not adapted, since $(f_1)^2\cdot x_2(x_0)=1\neq0$ and the second property of the adapted chart then fails. Following the constructive proof of Lemma \ref{th:adapted chart lemma}, one can find that the new coordinates $(y_1,y_2)$ defined by
\begin{equation*}
\left\{\begin{array}{l}
	y_1=x_1-\frac{x_1^2}{2}+x_2\\
	y_2=-\frac{x_1^2}{2}+x_2
\end{array}
\right.
\end{equation*}
give an adapted chart at $x_0$. In this coordinates the two vector fields are written as
$$f_1=\frac{\partial}{\partial y_1} \quad \mbox{ and } \quad f_0=\sin((y_1-y_2)^2)\left(\frac{\partial}{\partial y_1}+\frac{\partial}{\partial y_2}\right).$$
For $\epsilon>0$ the dilations defined in Definition \ref{th:dilation} are
\begin{equation}
\delta_\epsilon:(y_1,y_2)\mapsto (\epsilon y_1,\epsilon^4 y_2).
\label{eq:example4}
\end{equation}
Then the weights of the coordinate functions are $\calW(y_1)=1$ and $\calW(y_2)=4$, while the weights of the coordinate vector fields are $\calW(\frac{\partial}{\partial y_1})=1$ and $\calW(\frac{\partial}{\partial y_2})=4$.}

\emph{Finally, let us write the Taylor expansion of the two vector fields $f_1,f_0$:
\begin{eqnarray*}
f_1&=&\frac{\partial}{\partial y_1}\\
f_0&=&\left(y_1^2-2y_1y_2+y_2^2+o(|(y_1,y_2)|^2\right)\frac{\partial}{\partial y_2}
\end{eqnarray*}
We can see that $f_1$ has already weight $1$, while the only part of weight $2$ in $f_0$ is $y_1^2\frac{\partial}{\partial y_2}$. We can therefore conclude that the nilpotent approximation of $f_0,f_1$ is given by
$$\hat{f}_0=y_1^2\frac{\partial}{\partial y_2} \quad \mbox{ and } \quad \hat{f}_1=\frac{\partial}{\partial y_1}.$$}
\end{example}

\subsection{Order of the dilations}
We will analyze here the order of the dilations, that is the order of homogeneity of the volume form under the action of the dilations. This number will be crucial to find the order of degeneracy of the fundamental solution of the operator \eqref{eq:operator}. 

Let us consider the dilations $\delta_\epsilon$. They were defined by introducing the notation $x=(x^1,\ldots,x^m)$, where each component $x^i$ is a vector of length $k_i=d_i-d_{i-1}$. Then we set $\delta_\epsilon(x^1,x^2,\ldots,x^m)=(\epsilon x^1,\epsilon^2 x^2,\ldots,\epsilon^m x^m)$. Let $N$ be the order of homogeneity of the volume form $dx_1\wedge dx_2\wedge\ldots\wedge d x_n$ around the point $x_0$, that is a number such that
$$(\delta_\epsilon)_*(dx_1\wedge dx_2\wedge\ldots\wedge x_n)=\epsilon^N dx_1\wedge dx_2\wedge\ldots\wedge d x_n.$$
Then $N$ is given by
\begin{equation}
N:=\sum_{i=1}^m i \cdot k_i=\sum_{i=1}^m i \left(\dim L_i(x_0) -\dim L_{i-1}(x_0)\right).
\label{eq:N}
\end{equation}  
Since this number is very important we give here some examples to understand its meaning.

\begin{example}[Continuation of Example \ref{example1}]\label{th:example}
\emph{As a first example we consider the one given in Example \ref{example1}. We have already computed the filtration in the equations in \eqref{eq:example3}, so we already know that the integers $k_i:=\dim L_i-\dim L_{i-1}$ are
$$k_1=1\qquad k_2=k_3=0 \quad \mbox{ and } \quad k_4=1.$$
Therefore the order of the dilations is $N=1\cdot 1+4\cdot 1=5$, as one can compute directly from the explicit expression of the dilations in \eqref{eq:example4}. Notice that this number is much bigger then the dimension of the manifold.}
\end{example}

\begin{example}[Sub-Riemannian manifold]\label{th:sub-Riemannian}
\emph{Let us assume that the operator in \eqref{eq:operator} is induced by an equiregular sub-Riemannian manifold. In other words, we consider an operator without drift field, and the vector fields $f_1,\ldots,f_k$ generate a completely non-holonomic equiregular distribution, $\Delta$, of step $m$. Recall that the growth vector of the distribution is defined as the vector at any point $q$ of the manifold given by
$$(\Delta(q),\Delta^2(q),\ldots,\Delta^m(q)) \qquad \mbox{ where } \Delta^{i+1}:=[\Delta,\Delta^i].$$ 
Then the integers $k_i$ related to the filtration are the same defined by the growth vector, i.e. $k_i=\dim(\Delta_i)-\dim(\Delta_{i-1})$, and the number $N$ is exactly the homogeneous dimension $\calQ$ of the manifold. More explicitly
$$N=\calQ=1 \cdot k+2 \cdot k_2+\cdots m\cdot k_m=\sum_{i=1}^m i(\dim(\Delta_i)-\dim(\Delta_{i-1})).$$
In particular for the $n+1$ dimensional Heisenberg group the homogeneous dimension is $\calQ=n+2$.}
\end{example}

\begin{example}[Linear case]\label{th:involutive}
\emph{As a last example, we consider an involutive distribution $\calD$, spanned locally by $k$ constant vector fields, and assume that the drift field is linear in a neighborhood of $x_0$. Without loss of generality we can assume that 
$$f_i=\frac{\partial}{\partial x_i} \quad \forall 1\leq i\leq k\quad \mbox{ and } \quad f_0=\sum_{i,j=1}^n A_{ij}x_i\frac{\partial}{\partial x_j}$$
for some constants $A_{ij}$. Under these assumptions, the only Lie brackets different from zero are the one involving only one vector field of the distribution and the drift field. Let us call $A$ the $n\times n$ matrix with entries equal to $A_{ij}$ and $B$ the $n\times k$ matrix that is the identity in the first $k$ rows and is equal to zero in the last $n-k$ rows. Then H\"{o}rmander's condition of hypoellipticity \eqref{eq:hormander} becomes Kalman's condition of controllability for linear control systems, that is the following condition on the rank of Kalman's $n\times (nk)$ matrix
\begin{equation}
\mathrm{rank}[B,AB,A^2B,\ldots,A^{n-1}B]=n.
\label{eq:kalman}
\end{equation}
The filtration is then completely determined and we have
$$L_{2i-1}(x_0)=L_{2i}(x_0)=\mathrm{span}\{A^{j}B: 0\leq j\leq i-1\}.$$ 
Consequently the numbers $d_{2i-1}=d_{2i}$ are determined by the rank of the Kalman's matrix in \eqref{eq:kalman}, where we stop the series of matrices at $A^{i-1}B$. The numbers $k_j$ are zero if $j$ is even, while if $j=2i-1$ they are the number of new linearly independent columns obtained by adding the matrix  $A^{i-1}B$ to the previous one.
The step of the distribution is then an odd number $2\tilde{m}-1$ and $N$ is equal to an odd sum of integers:
$$N=\sum_{i=1}^{\tilde{m}}(2i-1)k_{2i-1}=1\cdot k_1+3\cdot k_3+5\cdot k_5+\cdots+(2\tilde{m}-1) k_{2\tilde{m}-1}.$$ 
This number is already appeared in literature, namely in Chapter 4 of \cite{article:AgrachevBarilariRizzi}, where Agrachev, Barilari and Rizzi compute the small time asymptotics of the cost functional associated to any ample, equiregular geodesic $\gamma$. In particular, if one computes the asymptotic of the cost functional associated to the ample, equiregular geodesics $\gamma(t)\equiv x_0$, then the trace of the principal term of the expansion is exactly the same term $N$. We refer to \cite{article:AgrachevBarilariRizzi} for an exhaustive presentation.}
\end{example}

%%%%%%%%%%%%%%%%%%%%%%%%%%%%%%%%%%%%%%%%%%%%%%%%%%%%%%%%%%%%%%%%%%%%%%%%%%%%%%%%%%%%%%%%%%%%%%%%%%%%%%%%%%%%%%%%%%%%%%%%%%%%%%%%%%%%%%%%%%%%%%%%%%%%%%%%%%%%%%%%%%%%%%%%%

\section{Small time asymptotic on the diagonal}\label{sec:5a}
We come back now to the perturbative method explained in Section \ref{sec:2}. Let 
$$\frac{\partial \varphi}{\partial t}-f_0(\varphi)-\frac{1}{2}\sum_{i=1}^k f_i^2(\varphi) \qquad \forall \varphi\in C^\infty(\bbR\times M)$$
be the differential operator \eqref{eq:operator} on $\bbR^+ \times M$ and assume that $f_0,f_1,\ldots,f_k$ satisfy H\"{o}rmander condition \eqref{eq:hormander}. Fix the graded structure, $(U,x,w)$, around $x_0$, induced by the filtration $\calL$ introduced in Section \ref{sec:graded} and let $p(t,x,y)$ be the fundamental solution of the partial differential equation corresponding to a volume form $\mu$, such that $\mu=dx_1\wedge\ldots\wedge dx_n$ in $U$. By Remark \ref{th:volume} this assumption on $\mu$ is not restrictive for our study. For $0<\epsilon<1$ we dilate the space around $x_0$ with the dilation $\delta_\epsilon$ defined in Definition \ref{th:dilation}, then by Proposition \ref{th:diffeo} the fundamental solution in $U$ of the operator
\begin{equation}
\frac{\partial }{\partial t}-\epsilon^2\left(\delta_{1/\epsilon*}f_0+\frac{1}{2}\sum_{i=1}^k \left(\delta_{1/\epsilon*}f_i\right)^2\right)
\label{eq:operator dilated}
\end{equation}
is
\begin{equation}
q_\epsilon(t,x,y)=\epsilon^Np(\epsilon^2 t,\delta_\epsilon x,\delta_\epsilon y) \qquad\quad \forall x,y\in U
\label{eq:q}
\end{equation}
where $N$ is the order of homogeneity of the volume form $\mu$ computed in \eqref{eq:N}.

The fields $f_0,f_1,\ldots,f_k$ can be written in terms of the nilpotent approximation as 
$$f_0=\hat{f}_0+g_0,\qquad\qquad f_i=\hat{f}_i+g_i \qquad 1\leq i\leq k,$$
where $g_0$ and $g_i$ are vector fields of order less then or equal to $1$ and $0$ respectively. Then the operator in \eqref{eq:operator dilated} can be decomposed in a principal part, perturbed by a small operator:
\begin{eqnarray*}
&&\frac{\partial }{\partial t}-\epsilon^2\left(\delta_{1/\epsilon*}f_0+\frac{1}{2}\sum_{i=1}^k \left(\delta_{1/\epsilon*}f_i\right)^2\right)=\\
&&=\frac{\partial }{\partial t}-\epsilon^2\left(\frac{1}{\epsilon^2}\hat{f}_0+o(\frac{1}{\epsilon^2})+\frac{1}{2}\sum_{i=1}^k \left(\frac{1}{\epsilon}\hat{f}_i+o(\frac{1}{\epsilon})\right)^2\right)\\
&&=\frac{\partial }{\partial t}\hat{f}_0-\frac{1}{2}\sum_{i=1}^k \hat{f}_i^2+o(1)\\
&&=\frac{\partial }{\partial t}-L_0-\calX_\epsilon
\end{eqnarray*}
where $\calX_\epsilon$ is an operator that goes to zero as $\epsilon$ and $L_0$ is the operator defined by the nilpotent approximation:
$$L_0:=\hat{f}_0+\frac{1}{2}\sum_{i=1}^k \hat{f}_i^2.$$

To this kind of operator we can apply Duhamel's formula \eqref{eq:duhamel asymptotic}, but to this end we have to prove that there exists the fundamental solution of the principal operator $\partial/\partial_t=L_0$. As we will prove now, this will follow by the property of hypoellipticity of the original operator, that are preserved by the nilpotent approximation, that defines $L_0$. The same statement can be found also in the paper by Bianchini and Stefani \cite{art:BianchiniStefani}.

\begin{proposition}
Let $\hat{f}_0,\hat{f}_1,\ldots,\hat{f}_k$ be the nilpotent approximation of the fields $f_0,f_1,\ldots,f_k$ defined in Definition \ref{th:nilpotent approximation}. Then
\begin{itemize}
	\item[(i)] for every bracket $\Lambda$ such that $||\Lambda||=i$, then $(\Lambda_f-\Lambda_{\hat{f}})\in L_{i-1}(x_0)$ and $\Lambda_{\hat{f}}(x_0)=0$ whenever $\Lambda_f(x_0)\in L_{i-1}(x_0)$, where $||\Lambda||$ denotes the weight of the bracket $\Lambda$ defined in Definition \ref{th:weight}.
	\item[(ii)] Assume $L_m(x_0)=\Lie_{x_0}\{f_0,f_1,\ldots,f_k\}$, then 
\begin{equation}
\Lie_{x_0}\{f_0,f_1,\ldots,f_k\}=\Lie_{x_0}\{\hat{f}_0,\hat{f}_1,\ldots,\hat{f}_k\}.
\label{eq:Lie}
\end{equation}
\end{itemize}
\end{proposition}
\begin{proof}
Let us prove the first statement. Let $(U,x)$ be coordinates around $x_0$ adapted to the filtration $\{L_i\}_i$. Then for every $i$, 
\begin{equation}
L_i(x_0)=\mathrm{span}\left\{\left.\frac{\partial}{\partial x_1}\right|_{x_0},\ldots,\left.\frac{\partial}{\partial x_{d_i}}\right|_{x_0}\right\}.
\label{eq:L_i}
\end{equation}
Let $\Lambda$ be a bracket such that $||\Lambda||=i$, then as proved in \cite{art:BianchiniStefani} Theorem 3.1, $\calO(\Lambda_f)\leq l_i$, where $\calO$ is the graded order associated to the graded structure induced by the filtration. Therefore there exist constants $a_j$ such that
\begin{equation}
\Lambda_f(x_0)=\sum_{j: l_j\leq l_i} a_j\left.\frac{\partial}{\partial x_{j}}\right|_{x_0}.
\label{eq:lambda}
\end{equation}
Notice that if two vector fields $h_1,h_2$ are homogeneous of graded order respectively $n_1$ and $n_2$, then their Lie bracket is either zero or homogeneous of order $n_1+n_2$. Then the Lie bracket $\Lambda_{\hat{f}}$ is either zero or homogeneous of order $||\Lambda||=i$. Therefore by equation \eqref{eq:lambda}, we have
\begin{equation}
\Lambda_{\hat{f}}(x_0)=\sum_{j: l_j= l_i} a_j\left.\frac{\partial}{\partial x_{j}}\right|_{x_0}.
\label{eq:lambda2}
\end{equation}
By subtracting \eqref{eq:lambda2} to \eqref{eq:lambda}, we find that  $(\Lambda_f-\Lambda_{\hat{f}})\in L_{i-1}(x_0)$, because $L_i(x_0)$ is obtained as in \eqref{eq:L_i}.

For the second statement, let $(U,x)$ be as before and let $\hat{V}(x_0)\in \Lie_{x_0}\{\hat{f}_0,\hat{f}_1,\ldots,\hat{f}_k\}$. Then $\hat{V}(x_0)=\Lambda_{\hat{f}}(x_0)$, for some bracket $\Lambda$ with $||\Lambda||=j$ equal to the graded order of $\hat{V}(x_0)$. Then by expression \eqref{eq:L_i}, there exist $\alpha_i$ such that $\hat{V}(x_0)=\sum_{i:l_i=j}\alpha_i \frac{\partial}{\partial x_i}$. Since $L_j(x_0)\subset L_m(x_0)=\Lie_{x_0}\{f_0,f_1,\ldots,f_k\}$, we have that $\hat{V}(x_0)\in \Lie_{x_0}\{f_0,f_1,\ldots,f_k\}$.

We prove the other inclusion by proving that $L_i(x_0)=\mathrm{span}\{\Lambda_{\hat{f}}(x_0):||\Lambda||\leq i\}$, for every $i$. We prove it by induction on $i$.

For $i=1$, $L_1(x_0)=\mathrm{span}\{\Lambda_{f}(x_0):||\Lambda||\leq 1\}$. Let $\Lambda_f(x_0)\in L_1(x_0)$, then by statement $(i)$, $(\Lambda_f-\Lambda_{\hat{f}})(x_0)\in L_0(x_0)=\{0\}$. Then $\Lambda_f(x_0)=\Lambda_{\hat{f}}(x_0)$ and the statement is true for $i=1$.

Assume that the statement is true for $i-1$, then $L_i(x_0)=\mathrm{span}\{\Lambda_{f}(x_0):||\Lambda||\leq i\}$ and $(\Lambda_f-\Lambda_{\hat{f}})(x_0)\in L_{i-1}(x_0)$. By the induction hypothesis, there exists $g\in \mathrm{span}\{\Lambda_{\hat{f}}(x_0):||\Lambda||\leq i-1\}$ such that
$$\Lambda_f(x_0)=\Lambda_{\hat{f}}(x_0)+g.$$
And the statement is proved also for $i$.

We conclude, since $\Lie_{x_0}\{f_0,f_1,\ldots,f_k\}=L_m(x_0)\subset \Lie_{x_0}\{\hat{f}_0,\hat{f}_1,\ldots,\hat{f}_k\}$.
\end{proof}

\begin{corollary}
The operator $\partial/\partial_t=L_0$ is hypoelliptic on $\bbR^n$.
\end{corollary}
\begin{proof}
By H\"{o}rmander's condition of hypoellipticity we know that $\Lie_{x_0}\{f_0,f_1,\ldots,f_k\}=\bbR^n$. Then the hypothesis of statement $(ii)$ of the Proposition are fulfilled and then also the nilpotent approximation is Lie bracket generating. Moreover, since $f_0(x_0)=0=\hat{f}_0(x_0)$, the approximation of the drift field can give some contribution in the generating process at $x_0$ only trough its Lie brackets with some other vector field. Then even H\"{o}rmander's condition of hypoellipticity \eqref{eq:hormander} holds at $x_0$ for the nilpotent approximation, that is
$$\mathrm{span}_{x_0}\{\frac{\partial}{\partial t}-\hat{f}_0,\hat{f}_1,\ldots,\hat{f}_k\}=\bbR^{n+1}.$$
Using the lower semi-continuity of the rank, we can then find a small neighborhood $U$ of $x_0$ where H\"{o}rmander condition holds at any point. 

Now by the homogeneity of the approximating system we know that
$$\delta_{\epsilon *}\hat{f}_0=\epsilon^2\hat{f}_0\qquad\mbox{ and }\qquad \delta_{\epsilon *}\hat{f}_i=\epsilon\hat{f}_i\quad \forall 1\leq i\leq k.$$
Therefore, since the differential operator commutes with the Lie brackets, we can extend H\"{o}rmander condition, which holds on a neighborhood of $x_0$, to the whole Euclidean space $\bbR^n$ and the operator $\partial/\partial_t=L_0$ is hypoelliptic on $\bbR^n$.
\end{proof}

\begin{remark}
\emph{
The assumption of $f_0(x_0)=0$ is necessaryfor the proof of this corollary. Indeed if $f_0(x_0)\neq 0$ it could be that $L$ is hypoelliptic, but $L_0$ is not.}

\emph{For example, on $\bbR^2$ the fields
$$f_1=\frac{\partial}{\partial x_1}\qquad f_0=(1+x_1)\frac{\partial}{\partial x_2}$$
satisfy H\"{o}rmander condition \eqref{eq:hormander}, but this fails for their nilpotent approximation
$$\hat{f}_1=\frac{\partial}{\partial x_1}\qquad \hat{f}_0=\frac{\partial}{\partial x_2}.$$
}
\end{remark}

We can conclude that the principal part of the operator in \eqref{eq:operator dilated} admits a well defined heat kernel, $q_0(t,x,y)$, for small time, that is given by the density function of the solution, $\xi(t)$, of the stochastic differential equation in Stratonovich form
\begin{eqnarray}
&&d\xi=\hat{f}_0(\xi) dt+\sum_{i=1}^k \hat{f}_i(\xi)\circ dw_i\label{eq:stochastic approx}\\
&&\xi(0)=x \nonumber
\end{eqnarray}
where $w_i$ is a 1-dim Brownian motion for every $1\leq i\leq k$. We can then apply the procedure introduced in Section \ref{sec:2} and we conclude by giving the asymptotic on the diagonal in $x_0$ of the fundamental solution $p(t,x,y)$.

\begin{theorem}\label{th:order}
Let $f_0,f_1,\ldots,f_k$ be smooth vector fields that satisfy H\"{o}rmander condition and $(U,x,w)$ be a graded coordinate neighborhood induced by the filtration $\calL$ defined in \eqref{eq:filtration} around a point $x_0$, where $f_0(x_0)=0$. Let $q_0(t,x,y)$ be the probability density function of the solution $\xi_t$ to equation \eqref{eq:stochastic approx} with initial condition $\xi(0)=x$ and assume that $q_0(1,x_0,x_0)$ is strictly positive. Then the short time asymptotic on the diagonal of the fundamental solution, $p(t,x,y)$, of the heat operator
\begin{equation}
\frac{\partial }{\partial t}-f_0-\frac{1}{2}\sum_{i=1}^k f_i^2
\label{eq:op}
\end{equation}
is given by
\begin{equation}
p(t,x_0,x_0)=\frac{q_0(1,x_0,x_0)}{t^{N/2}}(1+o(1)),
\label{eq:asymptotic theorem}
\end{equation}
where $N$ is the degree of homogeneity of the volume form $dx_1\wedge\ldots\wedge dx_n$ under the action of the dilations $\delta_\epsilon$ computed in \eqref{eq:N}.
\end{theorem}
\begin{proof}
By Duhamel's formula \eqref{eq:duhamel asymptotic} the asymptotic of the fundamental solution $q_\epsilon$ to the dilated operator as $\epsilon$ tends to 0 is
\begin{equation}
q_\epsilon(t,x,y)=q_0(t,x,y)+q_\epsilon*\calX_\epsilon q_0 (t,x,y),
\label{eq:q approx}
\end{equation}
provided that the remainder term $q_{\epsilon}*\calX_\epsilon q_0$ is negligible for $\epsilon$ small.

Let $\calZ$ be the operator $\frac{\calX_\epsilon}{\epsilon}$. We have to prove that
$$\int_0^1\int_U q_\epsilon(s,x_0,y)\calZ q_0(1-s,y,x_0)dy ds$$
is bounded uniformly in $\epsilon$. We split the integral between the integration on $0<s<\frac{1}{2}$ and the integration on $\frac{1}{2}<s<1$. For the first integral, since we are integrating on a bounded set where $1-s$ is far from zero, $\calZ q_0(1-s,y,x_0)$ is uniformly bounded by a constant $C$. Then
\begin{eqnarray*}
\int_0^{1/2}\int_U q_\epsilon(s,x_0,y)\calZ q_0(1-s,y,x_0)dy ds&\leq & C\int_0^{1/2}\int_U q_\epsilon(s,x_0,y) dy ds\\
&= &C \int_0^{1/2}\int_U \epsilon^N p(\epsilon^2 s,x_0,\delta_\epsilon y)dy ds\\
&\leq & \frac{C}{2}
\end{eqnarray*}
where the last inequality follows since $p$ has integral equal to $1$  over the whole space.

For the other part of the integral, the smooth function $q_\epsilon(s,x_0,y)$ is uniformly bounded on $U$, since $s$ is far from $0$. Moreover, the operator $\calZ$ is a combination of polynomials, that are bounded on the bounded set $U$ and of derivations of $q_0$ of any order. So let us control the integral
$$\int_{1/2}^1\int_U\frac{\partial}{\partial x_i} q_0(1-s,y,x_0) dy ds $$
for every $1\leq i\leq n$. By an integration by parts, it is equal to
$$\int_{1/2}^1\int_{U\cap\bbR^{n-1}} q_0(1-s,y_{U_1},x_0)- q_0(1-s,y_{U_2},x_0) dy_1\cdots \hat{dy_i}\cdots dy_n ds $$
where $y_{U_1}$ and $y_{U_2}$ are the two extremals  of the set $U$ on the line of integration in $y_i$, while $\hat{dy_i}$ denotes that we are not integrating in $dy_i$. This integral is bounded since $q_0$ has integral equal to $1$ on the whole space.

For the higher derivatives, let us assume by induction that for every set of indexes $I$, with $|I|=j$, we have proved that the integral
$$\int_{1/2}^1\int_U\frac{\partial^I}{\partial x_I} q_0(1-s,y,x_0) dy ds $$
is bounded. Let $J$ be a set of indexes such that $|J|=j+1$, and  that it differs from $I$ by an index $i$, then we can integrate by parts in the coordinate $x_i$ the whole integral
\begin{equation*}
\begin{split}
\int_{1/2}^1\int_U q_\epsilon(s,x_0,y) \frac{\partial^{J}}{\partial x_{J}} &q_0(1-s,y,x_0) dy ds =\\
&=\int_{1/2}^1\int_{U\cap\bbR_{n-1}} q_\epsilon(s,x_0,y) \left.\frac{\partial^{I}}{\partial x_{I}} q_0(1-s,y,x_0)\right|^{y_{U_1}}_{y_{U_2}} dy_1\cdots \hat{dy_i}\cdots dy_n ds - \\
&\qquad-\int_{1/2}^1\int_{U\cap\bbR_{n-1}}\int_{U\cap\bbR} \frac{\partial q_\epsilon}{\partial x_i}(s,x_0,y) \frac{\partial^{I}}{\partial x_{I}} q_0(1-s,y,x_0) dy ds 
\end{split}
\end{equation*}
and this is bounded by the induction hypothesis and since $q_\epsilon$ and $\frac{\partial q_\epsilon}{\partial x_i}$ are bounded away from $s=0$.

Now let us come back to the asymptotic in \eqref{eq:q approx}. By the definition of $q_\epsilon$, if in \eqref{eq:q} we fix $t=1$, $x=y=x_0$, and we let $\epsilon$ go to zero as $\epsilon=\sqrt{t}$, we notice that
$$p(t,x_0,x_0)=\frac{q_{\sqrt{t}}(1,x_0,x_0)}{\sqrt{t}^N}.$$
The desired small time asymptotic on the diagonal is then determined by the asymptotic \eqref{eq:q approx} for $q_\epsilon$ and we find
$$p(t,x_0,x_0)=\frac{q_0(1,x_0,x_0)}{t^{N/2}}(1+o(1)),$$
which is well defined since by hypothesis the leading term $q_0(1,x_0,x_0)$ doesn't vanish.
\end{proof}

\begin{remark}\label{th:remark1}
\emph{
The conclusion of Theorem \ref{th:order} holds also for the operator
\begin{equation}
\frac{\partial}{\partial t}-f_0-\frac{1}{2}\sum_{i=1}^k (f_i^2+(\diver_\mu f_i)f_i)
\label{eq:op div}
\end{equation}
where $\mu$ is any volume form on $M$ and $f_0,f_1,\ldots,f_k$ satisfy the assumptions of Theorem \ref{th:order}.
}

\emph{
Indeed we can treat the operator in \eqref{eq:op div} as an operator of the form
$$\frac{\partial}{\partial t}-\tilde{f}_0-\frac{1}{2}\sum_{i=1}^k f_i^2$$
where we take $\tilde{f}_0=f_0+\frac{1}{2}\sum_{i=1}^k (\diver_\mu f_i)f_i$, but in this case we could have $\tilde{f}_0(x_0)\neq 0$. However, if $f_0,f_1,\ldots,f_k$ satisfy H\"{o}rmander condition \eqref{eq:hormander}, then it is satisfied also by $\tilde{f}_0,f_1,\ldots,f_k$. Moreover, the filtration $\tilde{\cal{L}}$ generated by $\tilde{f}_0,f_1,\ldots,f_k$ is equal to the filtration $\cal{L}$ generated by $f_0,f_1,\ldots,f_k$, since $\tilde{f}_0$ differs from $f_0$ by a linear combination of the vector fields $f_1,\ldots,f_k$. Then the approximating system $\hat{\tilde{f}}_0, \hat{f}_1,\ldots,\hat{f}_k$ is the same as the approximating system $\hat{f}_0, \hat{f}_1,\ldots,\hat{f}_k$, hence the operators in \eqref{eq:op div} and in \eqref{eq:op} have the same principal part. The proof of Theorem \ref{th:order} for the operator in \eqref{eq:op div} now follows as the one for the operator \eqref{eq:op}.
}
\end{remark}

%%%%%%%%%%%%%%%%%%%%%%%%%%%%%%%%%%%%%%%%%%%%%%%%%%%%%%%%%%%%%%%%%%%%%%%%%%%%%%%%%%%%%%%%%%%%%%%%%%%%%%%%%%%%%%%%%%%%%%%%%%%%%%%%%%%%%%%%%%%%%%%%%%%%%%%%%%%%%%%%%%%%%%

\section{The principal operator and the associated control system}\label{sec:5}
In this section we are going to investigate the conditions for the positivity of the heat kernel, $q_0(t,x,y)$, that we have introduced in the last section. 

Let $f_0,f_1,\ldots,f_k$ satisfy H\"{o}rmander condition \eqref{eq:hormander} and consider the principal operator
\begin{equation}
\frac{\partial}{\partial t}-\hat{f}_0-\frac{1}{2}\sum_{i=1}^k \hat{f}_i^2
\label{eq:approx5}
\end{equation}
defined by the approximating system of the original system of vector fields. As already pointed out, it admits a smooth fundamental solution given by the probability density, $q_0(t,x,y)$, of the process $\xi_t$ to be at time $t$ in the point $y$ starting from the point $x$, where $\xi_t$ is the solution of the stochastic differential equation
\begin{equation}
d\xi_t=\hat{f}_0(\xi_t)dt+\sum_{i=1}^k \hat{f}_i(\xi_t)\circ dw_i.
\label{eq:sde}
\end{equation}

In their famous work \cite{art:stroockvaradhan} Stroock and Varadhan characterized the support of $q_0(t,x,y)$ and they showed that it is the set of reachable points from $x$ of the following associated control problem:
\begin{equation}
\dot{x}=\hat{f}_0(x)+\sum_{i=1}^k u_i\hat{f}_i(x)
\label{eq:control}
\end{equation}
where $x:[0,t]\rightarrow\bbR^n$ is a curve in $\bbR^n$ and $u=(u_1,\ldots,u_k)\in L^\infty([0,t];\bbR^k)$ are bounded controls.

Unfortunately, Stroock and Varadhan's result holds only for globally bounded vector fields, with bounded derivatives of any order. Since our vector fields are polynomial, they don't satisfy such assumptions and we can not directly apply the result of the support theorem. We will see in a moment how we can adapt their procedure to our system, but first we introduce two simple lemmas that will simplify our study.

\begin{lemma}\label{th:q positive}
Let $q_0(t,x,y)$ be the probability density of the solution to equation \eqref{eq:sde}. Then for every $\epsilon>0$, for every $t>0$ and for all $x\in \bbR^n$ it holds
$$q_0(t,x_0,x)=\epsilon^N q_0\left(\epsilon^2 t,x_0,\delta_{\epsilon}x\right).$$
In particular, we have the following equivalence
\begin{equation}
q_0(1,x_0,x_0)>0 \quad \Longleftrightarrow \quad q_0(t,x_0,x_0) \quad \forall t>0.
\label{eq:rescale q0}
\end{equation}
\end{lemma}
\begin{proof}
This is a corollary of Proposition \ref{th:diffeo}, since by definition of $\hat{f}_0,\hat{f}_1,\ldots,\hat{f}_k$ we have 
$$\epsilon^2 \delta_{\frac{1}{\epsilon}*}L_0=L_0.$$
\end{proof}

\begin{lemma}\label{th:neighbor}
Consider the control problem \eqref{eq:control} on $\bbR^n$. Let $y_1,y_2\in \bbR^n$ and $T>0$ be fixed and assume there exists a curve $y:[0,T]\rightarrow\bbR^n$ that satisfies the control problem \eqref{eq:control} for some control function $u\in L^\infty$ and such that $y(0)=y_1$ and $y(T)=y_2$.

Then for any $M>0$ the curve $x(t):=\delta_M \left(y\left(\frac{t}{M^2}\right)\right)$ is an admissible curve for the control problem defined on $[0,M^2T]$ with control $\tilde{u}(t):=\frac{1}{M}u\left(t/M^2\right)$, that connects $x_1:=\delta_M(y_1)$ with $x_2:=\delta_M(y_2)$ in time $M^2T$.

In particular, if the control problem \eqref{eq:control} is controllable in a neighborhood $U$ of $x_0$ in time $T$, then it is controllable in $\delta_M(U)$ in time $M^2T$.
\end{lemma}
\begin{proof}
The boundary conditions are easily satisfied, since $x(0)=\delta_M(y(0))=\delta_M(y_1)=x_1$ and $x(M^2T)=\delta_M(y(T))=\delta_M(y_2)=x_2$. Moreover, $x(t)$ is an admissible curve for the control problem with control $\frac{1}{M}u_i\left(t/M^2\right)$, indeed by the homogeneity of the approximating system we have
\begin{eqnarray*}
\dot{x}(t)&=&\frac{1}{M^2}\delta_{M*}\left(\dot{y}\left(t/M^2\right)\right)\\
&=&\frac{1}{M^2}\delta_{M*}\left[\hat{f}_0\left(y\left(t/M^2\right)\right)+\sum_{i=1}^k u_i\left(t/M^2\right)\hat{f}_i\left(y\left(t/M^2\right)\right)\right]\\
&=&\frac{1}{M^2}\left[M^2\hat{f}_0\left(y\left(t/M^2\right)\right)+\sum_{i=1}^k M u_i\left(t/M^2\right)\hat{f}_i\left(y\left(t/M^2\right)\right)\right]\\
&=&\hat{f}_0\left(x(t)\right)+\sum_{i=1}^k \frac{1}{M} u_i\left(t/M^2\right)\hat{f}_i\left(x(t)\right).
\end{eqnarray*}
\end{proof}

We generalize here Stroock and Varadhan's support theorem for the control problem defined by the approximating system. Indeed their theorem in \cite{art:stroockvaradhan} can be applied only to fields that are bounded and Lipschitz, with bounded derivatives of the first and second order, and does not hold for general unbounded fields. In our case, the fields are polynomial, so in general not globally bounded. However they are still very particular, because, since $\hat{f}_0$ has weight $2$ and $\hat{f_i}$ has weight $1$, for every $1\leq i\leq k$, every component of these fields depends only on the coordinates of less weight, i.e. $\hat{f}_i^{(j)}$ does not depend on $x_j,x_{j+1},\ldots,x_n$. This will allow us to modify the proof of the support theorem in order to extend it to our case.

\begin{definition}
Consider the control problem \eqref{eq:control} and let us call $x_u(t)$ the solution corresponding to a control $u$. The \emph{reachable set} of the control problem in time $t$ from $x$ is the set, 
$$\calA_t(x):=\left\{y\in \bbR^n\, : \exists u\in L^\infty([0,t];\bbR^k) \mbox{ such that } x_u(0)=x \mbox{ and } x_u(t)=y\right\}.$$
\end{definition}

\begin{lemma}\label{th:support}
Let $X_0,X_1,\ldots,X_k$ be smooth vector fields on $\bbR^n$, that satisfy H\"{o}rmander condition, and such that every $j$-th component $X_i^{(j)}$  of $X_i$, for $0\leq i\leq k$, does not depend on the coordinates $x_j,\ldots,x_n$, but only on the first coordinates $x_1,\ldots,x_{j-1}$. Let $\xi_t$ be the solution of the stochastic differential equation
$$d\xi_t=\tilde{X}_0(\xi_t)dt+\sum_{i=1}^k X_i(\xi_t) dw_t$$
where $\tilde{X}_0$ stands for the vector field whose $j$-component is given by 
$$\tilde{X}_0^{(j)}=X_0^{(j)}+\frac{1}{2}\sum_{i+1}^k\sum_{l=1}^nX_i^{(l)}\frac{\partial X_i^{(j)}}{\partial x_l}.$$ 
Let $p(t,x,y)$ be the probability density of $\xi_t$ to be in $y$ at time $t$ starting from the point $x$. Let $\calA_t(x)$ be the reachable set at time $t$ from $x$ of the associated control problem
\begin{equation}
\dot{x}=X_0(x)+\sum_{i=1}^k u_i(t)X_i(x)
\label{eq:dotx}
\end{equation}
where $u=(u_1,\ldots,u_k)\in L^\infty([0,t];\bbR^k)$ is a control function. Then 
$$\mathrm{supp}(p(t,x,\cdot))=\overline{\calA_t(x)}.$$
\end{lemma}

\begin{proof}
Stroock and Varadhan have proved this theorem under the assumption that the fields are Lipschitz and globally bounded, together with their derivatives of first and second order. Following their proof in \cite{art:stroockvaradhan}, we have to show that for a dense set of controls $u$ and $\forall \epsilon>0$
$$P_x(||\xi_t-x_t||<\epsilon):=P(\left.||\xi_t-x_t||<\epsilon\right|\xi_0=x)>0,$$
where $x_t$ is the solution of \eqref{eq:dotx} starting at $x$. In particular let us take $\psi\in C^2(\bbR^+;\bbR^k)$, with $\psi(0)=0$, and let $x_t$ be the solution of \eqref{eq:dotx} starting at $x$ with control $u_i(t):=\dot{\psi}_i(t)$.
%$$\dot{x}_t=X_0(x_t)+\sum_{i=1}^k \dot{\psi}_i(t) X_i(x_t).$$
Then for all $\epsilon>0$ we show that
\begin{equation}
P_x\left(\left.||\xi_t-x_t||<\epsilon\;\right|\;||w_t-\psi_t|| <\delta\right)\rightarrow 1
\label{eq:varadhan}
\end{equation}
as $\delta\searrow 0$. Indeed 
$$P_x\left(||\xi_t-x_t||<\epsilon\right)=P_x\left(\left.||\xi_t-x_t||<\epsilon\;\right|\;||w_t-\psi_t|| <\delta\right)\cdot P\left(||w_t-\psi_t|| <\delta\right)$$
and $P\left(||w_t-\psi_t|| <\delta\right)>0$ for every $\delta>0$.

Strooch and Varadhan proved \eqref{eq:varadhan} under the boundedness assumption that we don't have directly, but we will recover it by iterating a conditional probability. Indeed, notice that by our assumption the first component of every vector field, $X_i^{(1)}$, does not depend on any coordinate, so they are actually constant and they trivially satisfy Stroock and Varadhan's assumptions. Then the limit in \eqref{eq:varadhan} holds for the process $||\xi_{t}^{(1)}-x_{t}^{(1)}||$. 

Moreover let us assume, by induction, that the first $j-1$ components of $\xi_t$ live in a bounded set. Then the components $X_i^{(j)}$ are Lipschitz and bounded, together with their derivatives of any order, and we can apply Stroock and Varadhan's theorem to the $j$-th component of $\xi_t$ then
$$P_x\left(||\xi_{t}^{(j)}-x_{t}^{(j)}||<\epsilon\;|\;||w_t-\psi_t||<\delta,||\xi_{t}^{(l)}-x_{t}^{(l)}||<\epsilon\; \forall 1\leq l<j\right)\rightarrow 1$$
as $\delta \searrow 0$. 

The proof of \eqref{eq:varadhan} now follows using an iterated conditional probability, indeed in general for every measurable set $A_1,\ldots,A_n,B$, it holds
\begin{eqnarray*}
P\left(\left.\bigcap_{j=1}^n A_j\right\vert B\right)&=& P\left(\left.\bigcap_{j=2}^n A_j\right\vert B\cap A_1\right)P\left(A_1\vert B\right)\\
&=&P\left(\left.\bigcap_{j=3}^n A_j\right\vert B\cap A_1\cap A_2\right)P\left(A_2\vert B\cap A_1\right)P\left(A_1\vert B\right)\\
&\vdots&\\
&=&\prod_{j=1}^n P\left(A_j \left\vert B\cap \bigcap_{l=1}^{j-1} A_l\right.\right).
\end{eqnarray*}
Then
\begin{equation*}
\begin{split}
P&\left(||\xi_t-x_t||<\epsilon\;|\;||w_y-\psi_t||<\delta\right) =\\
&=\prod_{j=1}^n P\left(||\xi_{t}^{(j)}-x_{t}^{(j)}||<\epsilon\;|\;||w_t-\psi_t||<\delta,||\xi_{t}^{(l)}-x_{t}^{(l)}||<\epsilon\; \forall 1\leq l<j\right)\rightarrow 1
\end{split}
\end{equation*}
as
$\delta\searrow 0$ and we have proved \eqref{eq:varadhan} in our case.
\end{proof}

We are now ready to show a condition for the positivity of the fundamental solution of the approximating differential operator.

\begin{theorem}\label{th:positive}
Let $q_0(t,x,y)$ be the fundamental solution of \eqref{eq:approx5}. If the reachable set $\calA_t(x_0)$ of the associated control problem \eqref{eq:control} is a neighborhood of $x_0$ for some $t>0$, then $q_0(1,x_0,x_0)>0$.
\end{theorem}

\begin{proof}
By Lemma  \ref{th:q positive} it is enough to prove that $q_0(T,x_0,x_0)>0$ for some $T>0$. We will choose $T=2t$. Moreover, by Lemma \ref{th:neighbor}, if $\calA_t(x_0)$ is a neighborhood of $x_0$ for some $t>0$, it is a neighborhood for every $t>0$.

Assume by contradiction that $q_0(2t,x_0,x_0)=0$. By Chapman-Kolmogorov equation we know that
$$0=q_0(2t,x_0,x_0)=\int_{\bbR^n}q_0(t,x_0,y)q_0(t,y,x_0)dy=\int_{\overline{\calA_t(x_0)}}q_0(t,x_0,y)q_0(t,y,x_0)dy,$$
where we can restrict the space of integration, since by Lemma \ref{th:support}, $\mathrm{supp}(q_0(t,x_0,\cdot))=\overline{\calA_t(x_0)}$. Then for all $y\in\overline{\calA_t(x_0)}$ we have $q_0(t,y,x_0)=0$. 
The function $\tilde{q}(t,x,y):=q_0(t,y,x)$ is the fundamental solution of the adjoint operator to $L_0$, that is 
$$\frac{\partial}{\partial t}-L_0^*=\frac{\partial}{\partial t}-\hat{f}_0-\frac{1}{2} \sum_{i=1}^k \hat{f}_i^2.$$
Then $\tilde{q}(t,x,y)$ is the probability density function of the stochastic process $\tilde{\xi}_t$ solution of the stochastic equation
\begin{eqnarray*}
&&d\tilde{\xi}_t=-\hat{f}_0(\tilde{\xi}_t)dt+\sum_{i=1}^k \hat{f}_i(\tilde{\xi}_t)\circ dw_i(t)\\
&&\tilde{\xi}_0=x.
\end{eqnarray*}
By contradiction we have assumed that $\tilde{q}(2t,x_0,x_0)=0$, then again by Chapman-Kolmogorov equation we have
$$0=\tilde{q}(2t,x_0,x_0)=\int_{\tilde{\calA}_t(x_0)}\tilde{q}(t,x_0,y)\tilde{q}(t,y,x_0)dy$$
where $\tilde{\calA}_t(x_0)$ is the reachable set in time $t$ from the point $x_0$ of the associated control problem
$$\dot{x}=-\hat{f}_0(x)+\sum_{i=1}^k u_i(t)\hat{f}_i(x).$$
It follows that $q_0(t,x_0,y)=\tilde{q}(t,y,x_0)=0$ for all $y\in \tilde{\calA}_t(x_0)$. Since the point $x_0$ is a stationary point for the control problem, then $x_0\in \tilde{\calA}_t(x_0)$, for every $t>0$. By Krener theorem, $x_0$ is in the closure of $\mathrm{int}(\tilde{\calA}_t(x_0))$, then $\calA_t(x_0)\cap\tilde{\calA}_t(x_0)$ has non zero measure and $q_0(t,x_0,z)=0$ for all $z$ in this intersection. This is a contradiction to the support theorem.

We conclude that, if $\calA_t(x_0)$ is a neighborhood of $x_0$, then $q_0(t,x_0,x_0)>0$ for all $t>0$.
\end{proof}

\begin{remark}
\emph{In view of Theorem \ref{th:order} and Theorem \ref{th:positive} we can conclude the following properties about the asymptotic of the fundamental solution $p$ of the operator \eqref{eq:operator}.
\begin{itemize}
	\item[(i)] if the control problem associated with the original system:
	\begin{equation}
	\dot{x}=f_0(x)+\sum_{i=1}^k u_i(t) f_i(x)
	\label{eq:conclusion1}
	\end{equation}
	is not controllable around $x_0$, that is $\calA_t(x_0)$ is not a neighborhood of $x_0$, then 
	$$p(t,x_0,x_0)=0\qquad\qquad \forall t>0.$$ 
	Indeed, even if the support theorem can not be applied to this system, we still have that $\mathrm{supp}(p(t,x_0,\cdot))\subset \overline{\calA_t(x_0)}$, therefore $x_0$ is on the boundary or out of the support. Since $p(t,x_0,\cdot)$ is smooth, the conclusion follows.
	\item[(ii)]  if the control problem \eqref{eq:conclusion1} is controllable in $x_0$, then we study the controllability of its nilpotent approximation, defined by the fields $\hat{f}_0,\hat{f}_1,\ldots,\hat{f}_k$ introduced in Section \ref{sec:4}.
	\begin{itemize}
		\item[(ii.1)] if the approximating control problem
		\begin{equation}
		\dot{x}=\hat{f}_0(x)+\sum_{i=1}^k u_i(t) \hat{f}_i(x)
		\label{eq:conclusion2}
		\end{equation}
		is controllable around $x_0$, then the asymptotic is given in Theorem \ref{th:order}, where we see that the fundamental solution on the diagonal in $x_0$ blows up for small $t$ as the rational polynomial $\frac{a_0}{t^{N/2}}$, for a constant $a_0$ depending on the chosen volume and on the approximating system \eqref{eq:conclusion2}. The order $N$ is determined by the Lie algebra generated by the fields $f_0,f_1,\ldots,f_k$ at $x_0$ as explained in formula \eqref{eq:N}.
		\item[(ii.2)] it the approximating control problem \eqref{eq:conclusion2} is not controllable, then either $p(t,x_0,x_0)=0$ for every $t>0$ (but this is not possible if the fields $f_0,f_1,\ldots,f_k$ are bounded, with bounded derivatives, as explained in the proof of Theorem \ref{th:positive}), or $p(t,x_0,x_0)$ goes to infinity for small $t$ faster than $\frac{1}{t^{N/2}}$ as shown in the following Proposition. In \cite{article:BAL1} the authors show an example, where the asymptotic goes to infinity even exponentially fast.
	\end{itemize}
\end{itemize}}
\end{remark}

\begin{proposition}
Assume that the control problem \eqref{eq:conclusion1} is controllable around $x_0$, but the approximating control problem \eqref{eq:conclusion2} is not. If $p(t,x_0,x_0)>0$, then $\lim_{t\searrow 0}t^N p(t,x_0,x_0)=+\infty$, where $N$ is defined in \eqref{eq:N}.
\end{proposition}
\begin{proof}
To show this we need to prove that all the coefficients in the asymptotic \eqref{eq:asymptotic theorem} vanish. The asymptotic was found introducing the fundamental solutions $q_\epsilon$ and $q_0$ and by using Duhamel's formula \eqref{eq:duhamel asymptotic}. By iterating it we can achieve a better approximation of the asymptotic and find the higher coefficients. Indeed,we obtain
$$q_\epsilon=q_0+\sum_{i=1}^j q_0\left(*\calX_\epsilon q_0\right)^i+q_\epsilon\left(*\calX_\epsilon q_0\right)^{j+1}$$
for every $j\in\bbN$, where $\left(*\calX_\epsilon q_0\right)^i$ means that we iterate the convolution $i$ times. We have to show that all the terms $q_0\left(*\calX_\epsilon q_0\right)^i$ vanish at the point $(1,x_0,x_0)$ for every $i\in \bbN$. 

By Lemma \ref{th:support} we know that $q_0(t,x_0,y)=0$ for every $y\in \overline{\calA_t(x_0)}^c$ and all $t>0$. Let us assume by induction that 
\begin{equation}
q_0\left(*\calX_\epsilon q_0\right)^i(t,x_0,y)=0\qquad \forall y\in \overline{\calA_t(x_0)}^c, \; \forall t>0.
\label{eq:conclusion3}
\end{equation}
Then for every $y\in \overline{\calA_t(x_0)}^c$ we have
$$q_0\left(*\calX_\epsilon q_0\right)^{i+1}(t,x_0,y)=\int_0^t\int_{\overline{\calA_s(x_0)}}q_0\left(*\calX_\epsilon q_0\right)^i(s,x_0,z)\;\calX_\epsilon q_0(t-s,z,y)dz ds,$$
where the integral can be computed just on $\overline{\calA_s(x_0)}$ by the induction hypothesis. But $q_0(t-s,z,y)\equiv 0$ on $\overline{\calA_s(x_0)}$ by Chapman-Kolmogorov equation, then also all its derivatives vanish there. Then the integral is zero and we have proved \eqref{eq:conclusion3} for every $i$.

Since $q_0$ is smooth and $x_0$ is on the boundary of $\overline{\calA_t(x_0)}^c$, then equation \eqref{eq:conclusion3} holds also in the point $(t,x_0,x_0)$ for every $t>0$, that means that all the coefficients of the asymptotic \eqref{eq:asymptotic theorem} are zero.
\end{proof}

%%%%%%%%%%%%%%%%%%%%%%%%%%%%%%%%%%%%%%%%%%%%%%%%%%%%%%%%%%%%%%%%%%%%%%%%%%%%%%%%%%%%%%%%%%%%%%%%%%%%%%%%%%%%%%%%%%%%%%%%%%%%%%%%%%%%%%%%%%%%%%%%%%%%%%%%%%%%%%%%%%%%%%

\section{Examples}\label{sec:examples}
We end this paper with a study of some known examples, to understand better the meaning of Theorem \ref{th:order}.
\begin{example}[Continuation of Example \ref{example1}]
\emph{We complete the study of Example \ref{example1}. We have already computed the principal part of the dilated operator $L_\epsilon$, that  is
$$\frac{\partial}{\partial t}=\hat{f}_0+\frac{1}{2}\hat{f}_1^2=x_1^2\frac{\partial}{\partial x_2}+\frac{1}{2}\frac{\partial^2}{\partial x_1^2}.$$
This operator is indeed hypoelliptic since $\hat{f}_1=\frac{\partial}{\partial x_1}$ and $[\hat{f}_1,[\hat{f}_1,\hat{f}_0]]=2\frac{\partial}{\partial x_2}$ span the whole tangent space in every point. Let $q_0(t,x,y)$ be the density function of the solution $\xi(t)$ of the stochastic equation \eqref{eq:stochastic approx}, that in coordinates is given by
$$d\xi_1=dw_1,\qquad d\xi_2=x_1^2dt.$$
We can see that the second coordinate is actually deterministic and has positive derivative, so it can only increase. Consequently, if a path starts in $x_0=(0,0)$ the solution $\xi(t)$ will almost surely never come back to $x_0$ again, indeed $x_0$ is on the boundary of the support of $q_0(t,x_0,\cdot)$. Then the hypothesis of Theorem \ref{th:order} that requires $q_0(1,x_0,x_0)>0$ is not fullfilled.} 

\emph{Since the original control problem is controllable, this is an example of the type (ii.2).}
\end{example}

\begin{example}[Sub-Riemannian manifold: continuation of Example \ref{th:sub-Riemannian}]
\emph{The study of the asymptotic on the diagonal of the heat kernel on 3D contact sub-Riemannian manifolds has already been performed by Barilari in \cite{article:Barilari}. The nilpotent approximation of a sub-Riemannian manifold of dimension 3 is isometric to the Heisenberg group. Let us represent the Heisenberg group as $\bbR^3$ with coordinates $(x,y,z)$, then the approximating system can be written as
$$\hat{f}_1=\frac{\partial}{\partial x}+\frac{y}{2}\frac{\partial}{\partial z} \quad \mbox{ and } \quad \hat{f}_2=\frac{\partial}{\partial y}-\frac{x}{2}\frac{\partial}{\partial z}$$
As one can easily verify, the order of homogeneity of the volume form is given by $N=4$, as computed also with the general formula \eqref{eq:N}.}

\emph{The principal part of the dilated operator is hypoelliptic and symmetric, the associated control problem is then controllable, so there exists a well-defined symmetric heat kernel, that is positive for every $t>0$ as seen in Theorem \ref{th:positive}. The hypothesis of Theorem \ref{th:order} are then fullfilled and we find that the asymptotic on the diagonal of the original heat kernel $p(t,q,q')$ has the following order:
$$p(t,q,q)=\frac{a_0(q)+o(1)}{\sqrt{t}^4} \qquad \mbox{ for a smooth function }a_0(q)>0 \mbox{ on the manifold}.$$
This was the same order found by Barilari in \cite{article:Barilari}. }

\emph{This example is of the type (ii.1).}
\end{example}

\begin{example}[Ben Arous and L\'{e}andre]
\emph{We consider here an example studied by Ben Arous and L\'{e}andre in \cite{article:BAL1}. Consider the space $\bbR^2$ with coordinates $(x_1,x_2)$ and let 
$$f_0=x_1^a\frac{\partial}{\partial x_2}, \qquad f_1=\frac{\partial}{\partial x_1}, \qquad f_2=x_1^b\frac{\partial}{\partial x_2},$$
where $a$ and $b$ are positive integers. Then $x_0=(0,x_2)$ is a stationary point of the drift field for any $x_2\in\bbR$. The operator $L=f_0+\frac{1}{2}(f_1^2+f_2^2)$ satisfies even the strong H\"{o}rmander condition, i.e. the fields $f_1,f_2$ alone are Lie bracket generating.}

\emph{In \cite{article:BAL2}, it is shown the complete behaviour of the heat kernel $p(t,x,y)$ of this operator on the diagonal. We summarize here the most interesting properties for our study:
\begin{theorem}\label{th:Ben}
\begin{enumerate}
\item If $b\leq a+1$, then there exists a constant $K(a,b)>0$ such that
\begin{equation}
p(t,x_0,x_0)\sim \frac{K(a,b)}{\sqrt{t}^{b+2}}.
\label{eq:thm1}
\end{equation}
\item If $b> a+1$ and $a$ is even, then the fundamental solution $p(t,x_0,x_0)$ decreases with exponential velocity.%, more precisely
%$$\lim_{t\mapsto 0} t^{1-\frac{2}{b-a+1}}\log p(t,x_0,x_0)<0.$$
\end{enumerate}
\end{theorem}
The results found in this paper agree with the statement of the theorem. Indeed, let us consider the filtration $\calL$ given by the vector fields $f_0,f_1,\ldots,f_k$ at a point $x_0=(0,x_2)$. Let $m:=\min\{a+2; b+1\}$. It is easy to verify that the subspaces of the filtration $\calL$ are
$$L_i(x_0)=\left\{\begin{array}{ll}
\bbR\times\{0\} & \mbox{if } 1\leq i<m\\
\bbR^2 & \mbox{if } i=m.
\end{array}
\right.$$
Accordingly the coordinate $x_1$ has weight 1, while the coordinate $x_2$ has weight $m$ and the order of homogeneity of the volume form is given by
$$N=1+m=\left\{\begin{array}{ll}
b+2 & \mbox{if } b\leq a+1\\
a+3 & \mbox{if } b\geq a+1.
\end{array}
\right.$$
To determine the nilpotent approximation, it is convenient to divide the study in 3 cases, depending on the value of $a$ w.r.t. $b$.}

\emph{If $b< a+1$, then $m=b+1$. The nilpotent approximation is obtained by taking the Taylor expansion of the field $f_0$ of order $2$ and the Taylor expansion of order 1 of $f_1$ and $f_2$. Then we find that $\hat{f}_1=f_1$, $\hat{f}_2=f_2$ and $\hat{f}_0=0$. The principal part of the dilated operator is $\frac{1}{2}(\hat{f}_1^2+\hat{f}_2^2)$, that is hypoelliptic. Then there exists a well-defined heat kernel, $q_0(t,x,y)$, and, since the associated control system is controllable, $q_0(t,x_0,x_0)>0$ for every $t>0$. Then the hypothesis of Theorem \ref{th:order} are fulfilled and we find that the small time asymptotic  of the fundamental solution of $L$ has order $N/2=\frac{b+2}{2}$, that is exactly the one given in \eqref{eq:thm1}.}

\emph{If $b= a+1$, then $m=b+1$ and the nilpotent approximation is equal to the fields $f_1,f_2,f_0$ themselves. Since $f_1,f_2$ are Lie bracket generating the associated control system
$$\dot{x}=f_0+u_1 f_1+u_x f_2$$
is still controllable, then the heat kernel $q_0(t,x_0,x_0)$ is positive for every $t>0$, and we obtain again the statement \eqref{eq:thm1}.}

\emph{If $b> a+1$, then $\hat{f}_1=f_1$, $\hat{f}_0=f_0$ and $\hat{f}_2=0$. The principal operator $\hat{f}_0+\frac{1}{2}\hat{f}_1^2$ is still hypoelliptic, but if $a$ is even, then the heat kernel $q_0$ is zero in $x_0=(0,x_2)$ for any $t>0$. This is because a.e. path starting from $x_0$ will never come back to $x_0$ again, since the drift $f_0$ makes the first coordinate increase, if $x_1$ becomes different from 0. Then we can not apply Theorem \ref{th:order} and indeed Ben Arous and L\'{e}andre have shown an exponential decrease in this case.}
\end{example}

\bigskip
\noindent
{\bf Acknowledgments.}
{The author is grateful with Andrei Agrachev for many useful discussions and for introducing to the studied problem, and with Davide Barilari, for his interest in the subject and for many illuminating questions and remarks. The author has been partially supported by the Institut Henri Poincar\'{e}, Paris, where part of this research has been carried out.}

\bibliography{bibliografia}
\bibliographystyle{plain} 
\end{document}